\documentclass[12pt,a4paper]{amsart}
\usepackage{amssymb,eucal}
\usepackage[all,cmtip]{xy}

\calclayout

\newcommand{\+}{\nobreakdash-}
\renewcommand{\:}{\colon}

\newcommand{\rarrow}{\longrightarrow}

\newcommand{\ot}{\otimes}

\DeclareMathOperator{\Hom}{Hom}
\DeclareMathOperator{\Ext}{Ext}
\DeclareMathOperator{\Tor}{Tor}

\newcommand{\Bimod}{{\operatorname{\mathsf{--Bimod--}}}}
\newcommand{\Modl}{{\operatorname{\mathsf{--Mod}}}}
\newcommand{\Modr}{{\operatorname{\mathsf{Mod--}}}}

\newcommand{\nBimodn}{{\operatorname{\mathsf{--{}^n Bimod^n--}}}}
\newcommand{\Modln}{{\operatorname{\mathsf{--{}^n Mod}}}}
\newcommand{\Modrn}{{\operatorname{\mathsf{Mod^n--}}}}

\newcommand{\Modlz}{{\operatorname{\mathsf{--{}^0 Mod}}}}
\newcommand{\Modrz}{{\operatorname{\mathsf{Mod^0--}}}}

\newcommand{\tBimodt}{{\operatorname{\mathsf{--{}^t Bimod^t--}}}}
\newcommand{\Modlt}{{\operatorname{\mathsf{--{}^t Mod}}}}
\newcommand{\Modrt}{{\operatorname{\mathsf{Mod^t--}}}}

\newcommand{\sBimods}{{\operatorname{\mathsf{--{}^s Bimod^s--}}}}
\newcommand{\Modls}{{\operatorname{\mathsf{--{}^s Mod}}}}
\newcommand{\Modrs}{{\operatorname{\mathsf{Mod^s--}}}}

\newcommand{\Modlc}{{\operatorname{\mathsf{--{}^c Mod}}}}

\newcommand{\Ab}{\mathsf{Ab}}

\newcommand{\rop}{{\mathrm{op}}}
\newcommand{\sop}{{\mathsf{op}}}
\newcommand{\id}{{\mathrm{id}}}

\newcommand{\lrarrow}{\mskip.5\thinmuskip\relbar\joinrel\relbar
   \joinrel\rightarrow\mskip.5\thinmuskip\relax}

\newcommand{\bu}{{\text{\smaller\smaller$\scriptstyle\bullet$}}}

\newcommand{\C}{\mathcal C}
\newcommand{\bS}{\boldsymbol{\mathcal S}}

\newcommand{\sA}{\mathsf A}

\newcommand{\sC}{\mathsf C}
\newcommand{\sD}{\mathsf D}
\newcommand{\sE}{\mathsf E}
\newcommand{\sN}{\mathsf N}
\newcommand{\sP}{\mathsf P}
\newcommand{\sM}{\mathsf M}
\newcommand{\sS}{\mathsf S}
\newcommand{\sT}{\mathsf T}

\newcommand{\boZ}{\mathbb Z}
\newcommand{\boQ}{\mathbb Q}

\newcommand{\Section}[1]{\bigskip\section{#1}\medskip}
\theoremstyle{plain}
\newtheorem{thm}{Theorem}[section]
\newtheorem{prop}[thm]{Proposition}
\newtheorem{lem}[thm]{Lemma}
\newtheorem{cor}[thm]{Corollary}
\theoremstyle{definition}
\newtheorem{rem}[thm]{Remark}
\newtheorem{rems}[thm]{Remarks}
\newtheorem{ex}[thm]{Example}

\newtheorem{defn}[thm]{Definition}

\begin{document}

\title{Tensor-Hom formalism for modules \\ over nonunital rings}

\author{Leonid Positselski}

\address{Institute of Mathematics, Czech Academy of Sciences \\
\v Zitn\'a~25, 115~67 Prague~1 \\ Czech Republic} 

\email{positselski@math.cas.cz}

\begin{abstract}
 We say that a ring $R$ is t\+unital if the natural map
$R\ot_RR\rarrow R$ is an isomorphism, and a left $R$\+module $P$ is
c\+unital if the natural map $P\rarrow\Hom_R(R,P)$ is an isomorphism.
 For a t\+unital ring $R$, the category of t\+unital left $R$\+modules
is a unital left module category over an associative, unital monoidal
category of t\+unital $R$\+$R$\+bimodules, while the category of
c\+unital left $R$\+modules is opposite to a unital right module
category over the same monoidal category.
 Any left s\+unital ring $R$, as defined by Tominaga in 1975, is
t\+unital; and a left $R$\+module is s\+unital if and only if it is
t\+unital.
 For any (nonunital) ring $R$, the full subcategory of s\+unital
$R$\+modules is a hereditary torsion class in the category of
nonunital $R$\+modules; and for rings $R$ arising from small
preadditive categories, the full subcategory of c\+unital $R$\+modules
is closed under kernels and cokernels.
 However, over a t\+unital ring $R$, the full subcategory of t\+unital
modules need not be closed under kernels, and the full subcategory
of c\+unital modules need not be closed under cokernels in the category
of nonunital modules.
 Nevertheless, the categories of t\+unital and c\+unital left
$R$\+modules are abelian and naturally equivalent to each other;
they are also equivalent to the quotient category of the abelian
category of nonunital modules by the Serre subcategory of modules
with zero action of~$R$.
 This is a particular case of the result from a 1996 manuscript of
Quillen.
 We also discuss related flatness, projectivity, and injectivity
properties; and study the behavior of t\+unitality and c\+unitality
under the restriction of scalars for a homomorphism of nonunital rings.
 Our motivation comes from the theory of semialgebras over coalgebras
over fields.
\end{abstract}

\maketitle

\tableofcontents

\section*{Introduction}
\medskip

 Let $A$ be an associative, unital ring.
 Then the category of unital $A$\+$A$\+bimodules $A\Bimod A$ is
an associative, unital monoidal category with respect to
the operation~$\ot_A$ of tensor product over~$A$.
 The diagonal $A$\+$A$\+bimodule $A$ is the unit object of this monoidal
category.
 Furthermore, the category of unital left $A$\+modules $A\Modl$ is
a unital left module category over the monoidal category $A\Bimod A$,
while the category of unital right $A$\+modules $\Modr A$ is a unital
right module category over $A\Bimod A$.
 The same tensor product operation~$\ot_A$ provides these module
category structures.
 Finally, the tensor product of right and left $A$\+modules defines
a pairing between the right module category $\Modr A$ and the left
module category $A\Modl$ over the monoidal category $A\Bimod A$,
taking values in the category of abelian groups~$\Ab$.
 This structure is what we mean by the \emph{tensor formalism for
$A$\+modules}.
 We refer to the introduction to the book~\cite{Psemi} for a much
more general and abstract discussion of this formalism.

 Explicitly, the tensor formalism for $A$\+modules is expressed in
three natural isomorphisms of tensor products.
 For any right $A$\+module $N$, any $A$\+$A$\+bimodule $K$, and any
left $A$\+module $M$, one has a natural isomorphism of abelian groups
\begin{equation} \label{tensor-associativity}
 (N\ot_A K)\ot_A M\simeq N\ot_A(K\ot_AM).
\end{equation}
 For any left $A$\+module $M$, the natural left $A$\+module morphism
\begin{equation} \label{tensor-left-unitality}
 a\ot m\longmapsto am\:A\ot_AM\lrarrow M
\end{equation}
is an isomorphism.
 For any right $A$\+module $N$, the natural right $A$\+module morphism
\begin{equation} \label{tensor-right-unitality}
 n\ot a\longmapsto na\:N\ot_AA\lrarrow N
\end{equation}
is an isomorphism.

 Now let $R$ be a possibly nonunital, associative ring.
 Then the category of nonunital $R$\+$R$\+bimodules $R\nBimodn R$
is still an associative, unital monoidal category with respect to
the operation~$\ot_R$ of tensor product over~$R$.
 However, the $R$\+$R$\+bimodule $R$ is \emph{not} the unit object
of this monoidal category.

 To see what is going on, it suffices to consider the ring
$\widetilde R=\boZ\oplus R$ obtained by adjoining formally to $R$
a (new) unit together with its integer multiples.
 So the elements of $\widetilde R$ are the formal expressions $n+r$,
where $n\in\boZ$ and $r\in R$; and the multiplication in $\widetilde R$
is given by the obvious rule making $\boZ$ a subring in $\widetilde R$
and $R$ a two-sided ideal in~$\widetilde R$.
 Then the category of nonunital $R$\+$R$\+bimodules is naturally
equivalent to the category of unital
$\widetilde R$\+$\widetilde R$\+bimodules, and similarly for
the categories of left and right modules:
$$
 \Modrn R\simeq \Modr\widetilde R, \quad
 R\nBimodn R\simeq \widetilde R\Bimod\widetilde R, \quad
 R\Modln\simeq \widetilde R\Modl.
$$
 The tensor product operation~$\ot_R$ over the nonunital ring $R$
agrees with the tensor product~$\ot_{\widetilde R}$ over the unital
ring $\widetilde R$ under these category equivalences.
 The bottom line is that the $R$\+$R$\+bimodule $\widetilde R$ (rather
than the $R$\+$R$\+bimodule~$R$) is the unit object of the monoidal
category of nonunital bimodules $R\nBimodn R$.

 In this approach, the theory of nonunital rings and modules becomes
a particular case of the theory of unital ones.
 It is well-known, however, that there is a different point of view,
making the theory of modules over nonunital rings quite substantial.

 What one wants to do is to impose some kind of weak unitality
condition on the ring $R$, and particularly on the $R$\+modules.
 So, we are not primarily interested in arbitrary nonunital
$R$\+modules, but rather in $R$\+modules that are ``weakly unital''
in one sense or another.
 In this paper, we develop this approach aiming to obtain
an \emph{associative, unital monoidal category of weakly unital
$R$\+$R$\+bimodules} in which \emph{the $R$\+$R$\+bimodule $R$ would
be the unit object}.

 Alongside with the tensor formalism for modules, the \emph{Hom
formalism} plays an important role.
 Returning to our associative, unital ring $A$, let us consider
the opposite category $A\Modl^\sop$ to the category of
left $A$\+modules, and endow it with the right action of the monoidal
category $A\Bimod A$ defined by the rule
\begin{equation} \label{star-notation}
 P^\sop*K=\Hom_A(K,P)^\sop
\end{equation}
for all $P\in A\Modl$ and $K\in A\Bimod A$.
 This rule makes $A\Modl^\sop$ an associative, unital right module
category over $A\Bimod A$.
 Furthermore, the functor $(\Hom_A)^\sop\:A\Modl\times A\Modl^\sop
\rarrow\Ab^\sop$ defines a pairing between the right module
category $A\Modl^\sop$ and the left module category $A\Modl$ over
$A\Bimod A$, taking values in the opposite category $\Ab^\sop$ to
the category of abelian groups.
 This is what we call the \emph{Hom formalism for $A$\+modules}.
 Once again, we refer to the introduction to~\cite{Psemi} for
an abstract discussion.

 Explicitly, the Hom formalism for $A$\+modules is expressed in two
natural isomorphisms of Hom modules/groups.
 For any left $A$\+module $M$, any $A$\+$A$\+bimodule $K$, and any
left $A$\+module $P$, one has a natural isomorphism of abelian groups
\begin{equation} \label{Hom-associativity}
 \Hom_A(K\ot_AM,\>P)\simeq\Hom_A(M,\Hom_A(K,P)),
\end{equation}
which can be written as $P^\sop*(K\ot_AM)\simeq(P^\sop*K)*M$ in
the $*$\+notation of~\eqref{star-notation}.
 For any left $A$\+module $P$, the natural left $A$\+module morphism
\begin{equation} \label{Hom-unitality}
 p\longmapsto(a\mapsto ap)\:P\lrarrow\Hom_A(A,P)
\end{equation}
is an isomorphism.

 In addition to the tensor formalism, one would like also to extend
the Hom formalism to suitably defined ``weakly unital'' modules $P$
over some nonunital rings.
 Once again, both the tensor associativity
isomorphism~\eqref{tensor-associativity} and the Hom associativity
isomorphism~\eqref{Hom-associativity} hold without any unitality
assumptions on the rings and modules.
 It is the unitality isomorphisms~(\ref{tensor-left-unitality}\+-%
\ref{tensor-right-unitality}) and~\eqref{Hom-unitality} in
the tensor-Hom formalism that depend on the unitality assumptions on
the ring and modules and require some nontrivial generalizations in
the contexts of ``weakly unital'' modules.

 We refer to the very nice recent survey paper~\cite{Nys} for a general
discussion of various classes of (what we call) ``weakly unital'' rings
and modules and relations between them, containing many examples and
counterexamples.
 It is from the survey~\cite{Nys} that we learned about Tominaga's
1975 definition of \emph{s\+unital} rings and
modules~\cite{Tom1,Tom2,Tom3} playing an important role in this paper.

 Before we finish this introduction, let us briefly explain our
motivation.
 Let $k$~be a field, $R$ be an associative $k$\+algebra, and
$K\subset R$ be a subalgebra.
 Given a coassociative coalgebra $\C$ over~$k$ which is ``dual to
the algebra~$K$'' in some weak sense, under certain assumptions,
the construction of~\cite[Chapter~10]{Psemi} (for the base ring
$A=k$) produces a (semiassociative, semiunital) \emph{semialgebra}
$\bS=\C\ot_KR$ over the coalgebra~$\C$.
 The algebras $K$ and $R$ were, of course, assumed to be unital
in~\cite{Psemi}.
 In a forthcoming paper~\cite{Psa}, we intend to generalize
the construction of~\cite[Chapter~10]{Psemi} to certain nonunital
(but ``weakly unital'') algebras $K$ and~$R$.
 In this generalization, the coalgebra $\C$ is still presumed to be
counital.
 The point is that the semialgebra $\bS$ produced as the output of
the generalized construction remains semiunital.
 The present paper is intended to serve as a background reference
source for~\cite{Psa}.

 By far the most relevant reference in the context of the present
paper is an unpublished 1996 manuscript of Quillen~\cite{Quil}
(see also~\cite{Quil0}).
 The author learned about its existence on an advanced stage of
the present project; as a consequence, this paper contains our own
proofs of many results from~\cite{Quil} (with references given
for comparison wherever relevant).
 What we call \emph{t\+unital} modules are called ``firm modules''
in~\cite{Quil}, and what we call \emph{c\+unital} modules are called
``closed modules'' in~\cite{Quil}.
 Another important and relevant reference is the \emph{almost ring
theory}~\cite{GR}.

 There is a difference in motivation between~\cite{Quil} and
the present paper.
 It appears from~\cite[\S1]{Quil} that Quillen's aim was to construct
an abelian module category over a nonunital ring $R$ such that, when
$R$ is a unital ring, the category of unital $R$\+modules is produced
as the output.
 This problem is solved in~\cite{Quil} with the most well-behaved
solution obtained in the great generality of \emph{idempotent} rings
$R$, i.~e., rings $R$ such that $R^2=R$.
 On the other hand, the almost ring theory seems to be primarily
interested in commutative rings.
 Our task is to obtain a well-behaved associative, noncommutative
tensor-Hom formalism with the $R$\+$R$\+bimodule $R$ appearing as
the unit object of the monoidal category of ``weakly unital''
bimodules; and we solve this problem under more restrictive
assumptions of, at least, a \emph{t\+unital} ring~$R$.

\subsection*{Acknowledgement}
 I~am grateful to Jan \v St\!'ov\'\i\v cek and Michal Hrbek for
helpful discussions.
 The author is supported by the GA\v CR project 23-05148S and
the Czech Academy of Sciences (RVO~67985840).

\Section{t-Unital Rings and Modules}  \label{t-unital-secn}

 All \emph{rings} and \emph{modules} in this paper are associative,
but nonunital by default.
 The unitalization of a ring $R$ is denoted by
$\widetilde R=\boZ\oplus R$.
 So $R$ is a two-sided ideal in the unital ring $\widetilde R$ with
the unit element $1+0\in\boZ\oplus R$.

 The categories of nonunital left and right $R$\+modules are
denoted by $R\Modln$ and $\Modrn R$.
 Given two rings $R'$ and $R''$, the notation $R'\nBimodn R''$
stands for the category of nonunital $R'$\+$R''$\+bimodules.
 The undecorated notation $A\Modl$, $\Modr A$, and $A'\Bimod A''$
refers to the categories of unital (bi)modules and presumes the rings
$A$, $A'$, $A''$ to be unital.
 So there are natural equivalences (in fact, isomorphisms) of abelian
categories
\begin{equation} \label{nonunital-mods-as-mods-over-unitalization}
 \Modrn R\simeq\Modr\widetilde R, \quad
 R'\nBimodn R''\simeq \widetilde R'\Bimod\widetilde R'', \quad
 R\Modln\simeq\widetilde R\Modl.
\end{equation}

 Given a nonunital right $R$\+module $N$ and a nonunital left
$R$\+module $M$, the tensor product $N\ot_R M$ is defined in
the same way as for unital modules.
 Then there is a natural (identity) isomorphism of abelian groups
\begin{equation} \label{nonunital-and-unital-tensor-product}
 N\ot_R M\simeq N\ot_{\widetilde R} M.
\end{equation}
 For any right $R'$\+module $N$, any $R'$\+$R''$\+bimodule $B$, and
any left $R''$\+module $M$, there is a natural isomorphism of
abelian groups
\begin{equation} \label{nonunital-two-rings-tensor-associativity}
 (N\ot_{R'}B)\ot_{R''}M\simeq N\ot_{R'}(B\ot_{R''}M).
\end{equation}
 The associativity
isomorphism~\eqref{nonunital-two-rings-tensor-associativity} can be
obtained from the similar isomorphism for unital rings $\widetilde R'$
and $\widetilde R''$ and unital (bi)modules over them using the category
equivalences~\eqref{nonunital-mods-as-mods-over-unitalization} and
the isomorphism~\eqref{nonunital-and-unital-tensor-product}.

 For any left $R$\+module $M$, there is a natural left $R$\+module
morphism
\begin{equation} \label{tensor-left-unitality-map}
 r\ot m\longmapsto rm\: R\ot_R M\lrarrow M.
\end{equation}
 Similarly, for any right $R$\+module $N$, there is a natural right
$R$\+module morphism
\begin{equation} \label{tensor-right-unitality-map}
 n\ot r\longmapsto nr\: N\ot_R R\lrarrow N.
\end{equation}
 For the $R$\+modules $M=R=N$,
the maps~\eqref{tensor-left-unitality-map}
and~\eqref{tensor-right-unitality-map} agree with each other,
and provide an $R$\+$R$\+bimodule morphism
\begin{equation} \label{tensor-R-R-unitality-map}
 r'\ot r''\longmapsto r'r''\: R\ot_R R\lrarrow R.
\end{equation}

\begin{defn}
 We will say that a left $R$\+module $M$ is \emph{t\+unital} if
the natural map $R\ot_RM\rarrow M$ \,\eqref{tensor-left-unitality-map}
is an isomorphism.
 Similarly, a right $R$\+module $N$ is said to be \emph{t\+unital} if
the natural map $N\ot_R R\rarrow N$ \,\eqref{tensor-right-unitality-map}
is an isomorphism.
 An $R'$\+$R''$\+bimodule is said to be \emph{t\+unital} if it is
t\+unital both as a left $R'$\+module and as a right $R''$\+module.
 (What we call ``t\+unital modules'' are called ``firm modules''
in~\cite[Definition~2.3]{Quil}.)
\end{defn}

 It is clear from the preceding discussion of
the map~\eqref{tensor-R-R-unitality-map} that the left $R$\+module $R$
is t\+unital if and only if the right $R$\+module $R$ is t\+unital
(and if and only if the $R$\+$R$\+bimodule $R$ is t\+unital).
 Any one of these equivalent conditions holds if and only if
the natural map $R\ot_RR\rarrow R$ \,\eqref{tensor-R-R-unitality-map}
is an isomorphism.
 If this is the case, we will say that \emph{the ring $R$ is t\+unital}.

 Clearly, the map $R\ot_RR\rarrow R$ is surjective if and only if
the ring $R$ is idempotent, i.~e., $R^2=R$.
 In Example~\ref{idempotent-not-t-unital} below we construct
a counterexample of an idempotent, but not t\+unital ring~$R$,
i.~e., a ring $R$ for which the map $R\ot_RR\rarrow R$ is surjective,
but not injective.

\begin{lem} \label{t-unital-tensor-product-lemma}
\textup{(a)} Let $B$ be an $R'$\+$R''$\+bimodule and $M$ be
an left $R''$\+module.
 Assume that $B$ is t\+unital as a left $R'$\+module.
 Then the left $R'$\+module $B\ot_{R''}M$ is also t\+unital. \par
\textup{(b)} Let $N$ be a right $R'$\+module and $B$ be
an $R'$\+$R''$\+bimodule.
 Assume that $B$ is t\+unital as a right $R''$\+module.
 Then the right $R''$\+module $N\ot_{R'}B$ is also t\+unital.
\end{lem}

\begin{proof}
 Follows from the associativity
isomorphism~\eqref{nonunital-two-rings-tensor-associativity} and
commutativity of the suitable triangular diagram formed by
the associativity isomorphism together with the unitality comparison
maps~\eqref{tensor-left-unitality-map}
or~\eqref{tensor-right-unitality-map}.
\end{proof}

 The following corollary is a particular case
of~\cite[Proposition~4.4]{Quil}.

\begin{cor} \label{t-unitality-created-by-tensor-with-R}
 Let $R$ be a t\+unital ring.  Then \par
\textup{(a)} for any left $R$\+module $M$, the left $R$\+module
$R\ot_RM$ is t\+unital; \par
\textup{(b)} for any right $R$\+module $N$, the right $R$\+module
$N\ot_RR$ is t\+unital.
\end{cor}

\begin{proof}
 Take $B=R=R'=R''$ in Lemma~\ref{t-unital-tensor-product-lemma}.
\end{proof}

 Given three rings $R$, $R'$, $R''$, we will denote by $R\Modlt$,
$\Modrt R$, and $R'\tBimodt R''$ the full subcategories of
t\+unital left $R$\+modules, t\+unital right $R$\+modules, and
t\+unital $R'$\+$R''$\+bimodules in the abelian categories
$R\Modln$, $\Modrn R$, and $R'\nBimodn R''$, respectively.
{\hbadness=1325\par}

\begin{cor} \label{t-unital-monoidal-module-categories}
 Let $R$ be a t\+unital ring.
 Then the additive category of t\+unital $R$\+$R$\+bimodules
$R\tBimodt R$ is an associative and unital monoidal category, with
the unit object $R\in R\tBimodt R$, with respect to the tensor
product operation\/~$\ot_R$.
 The additive category of t\+unital left $R$\+modules $R\Modlt$ is
an associative and unital left module category over $R\tBimodt R$,
and the additive category of t\+unital right $R$\+modules $\Modrt R$
is an associative and unital right module category over $R\tBimodt R$,
with respect to the tensor product operation\/~$\ot_R$.
\end{cor}

\begin{proof}
 Follows from the associativity
isomorphism~\eqref{nonunital-two-rings-tensor-associativity}
and Lemma~\ref{t-unital-tensor-product-lemma} (for $R'=R=R''$).
\end{proof}

\begin{lem} \label{t-unital-modules-closure-properties}
 Let $R$ be an arbitrary (nonunital) ring.
 Then the full subcategory of t\+unital left $R$\+modules
$R\Modlt\subset R\Modln$ is closed under extensions and all colimits
(including cokernels, direct sums, and direct limits) in the abelian
category of nonunital modules $R\Modln$.
\end{lem}

\begin{proof}
 This is our version of~\cite[Proposition~4.2]{Quil}.
 Notice that it is \emph{not} claimed in the lemma that $R\Modlt$ is
closed under quotient modules in $R\Modln$, but only under cokernels.
 The closedness under arbitrary colimits follows immediately from
the fact that the tensor product functor $R\ot_R{-}$ preserves colimits
(which can be proved in the way similar to the unital case, or deduced
from the unital case using
the isomorphism~\eqref{nonunital-and-unital-tensor-product}).

 To prove the closedness under extensions, consider a short exact
sequence $0\rarrow K\rarrow L\rarrow M\rarrow0$ in $R\Modln$ with
$K$, $M\in R\Modlt$.
 We have a commutative diagram
$$
 \xymatrix{
  & R\ot_RK \ar[r]\ar@{=}[d] & R\ot_R L \ar[r]\ar[d]
  & R\ot_RM \ar[r]\ar@{=}[d] & 0 \\
  0 \ar[r] & K \ar[r] & L \ar[r] & M \ar[r] & 0
 }
$$
where the sequence in the upper row is right exact and the sequence in
the lower row is short exact.
 The rightmost and leftmost vertical maps are isomorphisms
by assumption.
 Now the composition $R\ot_RK\rarrow K\rarrow L$ is injective, since
the map $K\rarrow L$ is injective.
 It follows that the map $R\ot_RK\rarrow R\ot_RL$ is injective as well.
 So the upper row is, in fact, a short exact sequence, and the middle
vertical map $R\ot_RL\rarrow L$ is an isomorphism by the 5\+lemma.
\end{proof}

 One shortcoming of the theory developed in this section is that
the full subcategory $R\Modlt$ \emph{need not} be closed under kernels
in $R\Modln$, as explained in the following remark and example.
 A way to resolve this problem will be suggested in the next
Section~\ref{s-unital-secn}, where we will consider the more restrictive
\emph{s\+unital} setting.

\begin{rem} \label{Tor-perpendicularity-remark}
 The construction of the full subcategory $R\Modlt\subset R\Modln$ can
be viewed as a particular case of the construction of
the \emph{$\Tor_{0,1}$\+perpendicular subcategory} in the spirit
of~\cite[Proposition~1.1]{GL} or~\cite[Theorem~1.2(b)]{Pcta}.

 Let us consider the ring $R$ as an ideal in the unital ring
$\widetilde R=\boZ\oplus R$, as above.
 Then the quotient ring $\boZ=\widetilde R/R$ becomes a unital left
and right $\widetilde R$\+module, or equivalently, a nonunital
left and right $R$\+module with the zero action of~$R$.
 Given a nonunital left $R$\+module $M$, the natural map
$R\ot_R M=R\ot_{\widetilde R}M\rarrow\widetilde R\ot_{\widetilde R}M
=M$ is an isomorphism \emph{if and only if} one has
$\widetilde R/R\ot_{\widetilde R}M=0=
\Tor_1^{\widetilde R}(\widetilde R/R,M)$.
 (This observation is a particular case of the discussion
in~\cite[Definition~2.3]{Quil}.)

 Now the mentioned results from~\cite{GL,Pcta} tell that the full
subcategory $R\Modlt$ is closed under kernels in $R\Modln$ (and,
consequently, abelian) whenever the flat dimension of the unital right
$\widetilde R$\+module $\widetilde R/R$ does not exceed~$1$, i.~e.,
in other words, the unital $\widetilde R$\+module $R$ is flat.
 This can be also easily seen directly from the definition: if
the functor $R\ot_R{-}\,\:R\Modln\rarrow R\Modln$ is exact (preserves
kernels), then the full subcategory $R\Modlt\subset R\Modln$ is
closed under kernels.
 The next example shows that this does \emph{not} hold in general.
\end{rem}

\begin{ex} \label{two-variables-non-t-unital-kernel-counterex}
 Let $\widetilde R=\boZ[x,y]$ be the commutative ring of polynomials
in two variables~$x$, $y$ with integer coefficients.
 Let $R\subset\widetilde R$ be the ideal $R=(x,y)\subset\widetilde R$
spanned by the elements $x$, $y\in~\widetilde R$.
 Then it is clear that the ring $\widetilde R$ is the unitalization
of the ring $R$, just as the notation suggests.
 The flat dimension of the unital $\widetilde R$\+module
$\boZ=\widetilde R/R$ is equal to~$2$, and the flat dimension of
the unital $\widetilde R$\+module $R$ is equal to~$1$.
 We claim that the full subcategory of t\+unital $R$\+modules
$R\Modlt$ is \emph{not} closed under kernels in $R\Modln=
\widetilde R\Modl$.

 Indeed, consider the generalized Pr\"ufer module
$P_x=\boZ[x,x^{-1}]/\boZ[x]$ over the ring $\boZ[x]$, and the similar
module $P_y=\boZ[y,y^{-1}]/\boZ[y]$ over the ring~$\boZ[y]$.
 Furthermore, consider the $\boZ[x,y]$\+module $P_x\ot_\boZ P_y$.
 Then one can easily see that $\Tor^{\boZ[x]}_0(\boZ,P_x)=0$
and $\Tor^{\boZ[x]}_1(\boZ,P_x)=\boZ$ (where $x$~acts by zero in
the $\boZ[x]$\+module~$\boZ$).
 Consequently, by the K\"unneth formula (see
Lemma~\ref{kuenneth-formula} below),
\,$\Tor_n^{\widetilde R}(\boZ,\>P_x\ot_\boZ P_y)=0$ for $n=0$,~$1$,
and $\boZ$ for $n=2$.
 Following Remark~\ref{Tor-perpendicularity-remark}, we have
$P_x\ot_\boZ P_y\in R\Modlt$.

 On the other hand, denote by $Q_y$ the $\boZ[y]$\+module $P_y$ viewed
as a $\boZ[x,y]$\+module with the given (Pr\"ufer) action of~$y$
and the zero action of~$x$.
 Then $\Tor_1^{\widetilde R}(\boZ,Q_y)=\Tor_1^{\boZ[y]}(\boZ,P_y)=\boZ$.
 So $Q_y\notin R\Modlt$.
 Nevertheless, $Q_y$ is the kernel of the (surjective) $R$\+module
(or $\widetilde R$\+module) map $x\:P_x\ot_\boZ P_y\rarrow
P_x\ot_\boZ P_y$.

 Notice, however, that the ring $R$ in this example is \emph{not}
t\+unital.
 In fact, $R$ is not even an idempotent ideal in the unital
ring~$\widetilde R$, i.~e., $R^2\ne R$.
 An example of t\+unital ring $T$ for which the full subcategory
$T\Modlt$ is not closed under kernels in $T\Modln$ will be given
below in Example~\ref{t-unital-ring-non-t-unital-kernel-counterex}.
\end{ex}

 For any t\+unital ring $R$,
Theorem~\ref{main-two-category-equivalences-thm}(a) below tells that
the category $R\Modlt$ is abelian.
 On the other hand, again over a t\+unital ring~$R$, according to
Corollary~\ref{closed-under-kernels-iff-flat}, the full subcategory
of t\+unital modules $R\Modlt$ is closed under kernels in $R\Modln$
\emph{if and only if} $R$ is a flat unital right $\widetilde R$\+module.

\begin{rem}
 Speaking of unital modules over a unital ring $A$ (such as
$A=\widetilde R$), we often emphasize that we work in the category
of unital $A$\+modules (e.~g., that particular modules are flat or
projective \emph{as unital modules}), but it actually does not
matter that much.
 Nonunital modules over any ring $R$ are the same things as unital
modules over the unitalization $\widetilde R$ of the ring $R$,
and the unitalization $\widetilde A=\boZ\oplus A$ of a \emph{unital}
ring $A$ is naturally isomorphic to the Cartesian product of
the rings $\boZ$ and $A$, that is $\widetilde A\simeq\boZ\times A$.
 The ring isomorphism $\boZ\oplus A\rarrow\boZ\times A$ takes a formal
sum $n+a$ with $n\in\boZ$ and $a\in A$ to the pair $(n,\>ne+a)$,
where $e\in A$ is the unit element.
 Accordingly, the abelian category of nonunital $A$\+modules is just
the Cartesian product of the abelian categories of abelian groups and
unital $A$\+modules,
$$
 A\Modln\simeq \widetilde A\Modl\simeq
 (\boZ\times A)\Modl\simeq\boZ\Modl\times A\Modl.
$$
 So there is not much of a difference whether to consider unital
modules over unital rings as objects of the category of unital or
nonunital modules over such ring.
 (Cf.\ the discussion in~\cite[\S1]{Quil}, which emphasizes
a different point.)

 In particular, all the groups/modules $\Tor$ and $\Ext$ in this paper
are taken over unital rings and in the categories of unital modules.
\end{rem}

\Section{s-Unital Rings and Modules}  \label{s-unital-secn}

 The following definition is due to Tominaga~\cite{Tom1,Tom2,Tom3}.
 The exposition in the survey paper~\cite{Nys} is very much recommended.

\begin{defn}
 Let $R$ be a (nonunital) ring.
 A left $R$\+module $M$ is said to be \emph{s\+unital} if for every
element $m\in M$ there exists an element $e\in R$ such that $em=m$
in~$M$.
 (Notice that \emph{no} assumption of idempotency of~$e$ is made here!)
 Similarly, a right $R$\+module $N$ is said to be \emph{s\+unital} if
for every element $n\in N$ there exists an element $e\in R$ such that
$ne=n$ in~$N$.
 Given two rings $R'$ and $R''$, an $R'$\+$R''$\+bimodule is said to be
\emph{s\+unital} if it is s\+unital both as a left $R'$\+module and
as a right $R''$\+module.
\end{defn}

 Given three rings $R$, $R'$, $R''$, we will denote by $R\Modls$,
$\Modrs R$, and $R'\sBimods R''$ the full subcategories of
s\+unital left $R$\+modules, s\+unital right $R$\+modules, and
s\+unital $R'$\+$R''$\+bimodules in the abelian categories
$R\Modln$, $\Modrn R$, and $R'\nBimodn R''$, respectively.
{\hbadness=1300\par}

\begin{prop} \label{s-unital-hereditary-torsion}
 Let $R$ be an arbitrary (nonunital) ring.
 Then the full subcategory of s\+unital $R$\+modules $R\Modls\subset
R\Modln$ is closed under submodules, quotients, extensions, and all
colimits (including direct sums and direct limits) in the abelian
category of nonunital modules $R\Modln$.
 In other words, $R\Modls$ is a \emph{hereditary torsion class} in
$R\Modln$ in the sense of\/~\cite{Dic}, \cite[Sections~VI.2\+-3]{Sten}.
\end{prop}

\begin{proof}
 Closedness under submodules, quotients, and direct limits (indexed by
directed posets) is obvious immediately from the definition.
 All direct sums are direct limits of finite direct sums, and finite
direct sums are finitely iterated extensions.
 So it suffices to prove the closedness under extensions.

 Let $0\rarrow K\rarrow L\rarrow M\rarrow0$ be a short exact sequence
of nonunital left $R$\+modules.
 Assume that the $R$\+modules $K$ and $M$ are s\+unital.
 Given an element $l\in L$, we need to find an element $g\in R$ such
that $gl=l$ in~$L$.
 Since the $R$\+module $M=L/K$ is s\+unital, there exists an element
$e\in R$ such that $e(l+K)=l+K$ in~$M$.
 So we have $el-l\in K$.
 Since the $R$\+module $K$ is s\+unital, there exists an element
$f\in R$ such that $f(el-l)=el-l$ in~$K$.
 Now we have $(e+f-fe)l=l$ in~$L$, and it remains to put
$g=e+f-ef\in R$.
\end{proof}

 The following result goes back to Tominaga~\cite{Tom2}.
 It can be also found in~\cite[Proposition~2.8 of the published
version or Proposition~8 of the \texttt{arXiv} version]{Nys}.

\begin{cor}[Tominaga~{\cite[Theorem~1]{Tom2}}]
\label{simultaneous-for-several}
 Let $R$ be a ring and $M$ be an s\+unital left $R$\+module.
 Then, for any finite collection of elements $m_1$,~\dots, $m_n\in M$,
there exists an element $e\in R$ such that $em_i=m_i$ for all\/
$1\le i\le n$.
\end{cor}

\begin{proof}
 By Proposition~\ref{s-unital-hereditary-torsion}, the $R$\+module
$M^n=\bigoplus_{i=1}^n M$ is s\+unital.
 By the definition, this means existence of the desired element
$e\in R$ such that $em=m$ for the element $m=(m_1,\dotsc,m_n)\in M^n$.
\end{proof}

 We will not use the following lemma in this paper, but we include it
for the sake of completeness of the exposition.

\begin{lem}
 Let $R$ be a ring and $B$ be an s\+unital $R$\+$R$\+bimodule.
 Then, for any finite collection of elements $b_1$,~\dots, $b_n\in B$,
there exists an element $e\in R$ such that $eb_i=b_i=b_ie$ for all\/
$1\le i\le n$.
\end{lem}

\begin{proof}
 This is~\cite[Proposition~2.10 of the published
version or Proposition~10 of the \texttt{arXiv} version]{Nys}.
\end{proof}

\begin{defn}[Tominaga~\cite{Tom1,Tom2}]
 A ring $R$ is said to be \emph{left s\+unital} if $R$ is an s\+unital
left $R$\+module.
 Similarly, $R$ is said to be \emph{right s\+unital} if it is
s\+unital as a right $R$\+module.
 A ring $R$ is \emph{s\+unital} if it is both left and right s\+unital.
\end{defn}

 The following theorem is due to Tominaga~\cite[Remark~(2)
in Section~1 on p.~121\+-122]{Tom2}.

\begin{thm}[\cite{Tom2}] \label{main-s-unital-tensor-product-theorem}
\textup{(a)} Let $R$ be a left s\+unital ring.
 Then, for any s\+unital left $R$\+module $M$, the natural map
$R\ot_R M\rarrow M$ \,\eqref{tensor-left-unitality-map} is
an isomorphism.
 In other words, any s\+unital left $R$\+module is t\+unital. \par
\textup{(b)} Let $R$ be a right s\+unital ring.
 Then, for any s\+unital right $R$\+module $N$, the natural map
$N\ot_R R\rarrow N$ \,\eqref{tensor-right-unitality-map} is
an isomorphism.
 In other words, any s\+unital right $R$\+module is t\+unital.
\end{thm}

\begin{proof}
 Let us prove part~(a).
 For this purpose, we will construct an inverse map
$$
 \phi_M\:M\lrarrow R\ot_R M
$$
to the multiplication map~\eqref{tensor-left-unitality-map}.
 Given an element $m\in M$, pick an arbitrary element $e\in R$
such that $em=m$ in $M$, and put $\phi(m)=e\ot m\in R\ot_RM$.

 The key step is to check that the map~$\phi$ is well-defined.
 The assumption of left s\+unitality of the ring $R$ is used here.

 Let $f\in R$ be another element such that $fm=m$ in~$M$.
 Since the left $R$\+module $R$ is s\+unital,
Corollary~\ref{simultaneous-for-several} (for $n=2$) tells that
there exists an element $g\in R$ such that $ge=e$ and $gf=f$
in~$R$.
 Now we have
$$
 e\ot m=ge\ot m=g\ot em=g\ot m=g\ot fm=gf\ot m=f\ot m
$$
in $R\ot_RM$.
 Here the first and sixth equations hold by the choice of
the element~$g$, the second and fifth equations are provided
by the definition of the tensor product, and the third and fourth
equations hold by the conditions imposed on the elements~$e$ and~$f$.
 Thus the map~$\phi$ is well-defined.

 To check that the map~$\phi$ is additive (i.~e., an abelian
group homomorphism), we apply Corollary~\ref{simultaneous-for-several}
(for $n=2$) to the s\+unital $R$\+module~$M$.
 Given two elements $m'$ and $m''\in M$, the corollary tells that
there exists an element $e\in R$ such that $em'=m'$ and $em''=m''$
in~$M$.
 Hence $e(m'+m'')=m'+m''$.
 Now it is clear that $\phi(m'+m'')=e\ot(m'+m'')=e\ot m'+e\ot m''
=\phi(m')+\phi(m'')$.

 One can also check directly that $\phi$~is a left $R$\+module map
(though we will not need this).
 Given any elements $r\in R$ and $m\in M$, choose an element $f\in R$
such that $fr=r$ in~$R$.
 Then $frm=rm$ in~$M$.
 Also, choose an element $e\in R$ such that $em=m$ in~$M$.
 Now we have $\phi(rm)=f\ot rm=fr\ot m=r\ot m=r\ot em=re\ot m=
r\phi(m)$ in $R\ot_RM$.

 What we need to check, though, is that both the compositions
$M\rarrow R\ot_RM\rarrow M$ and $R\ot_RM\rarrow M\rarrow R\ot_RM$
are identity maps.
 In the case of the former composition, this is immediate from
the constructions.
 Concerning the latter composition, we know that both the maps being
composed are abelian group homomorphisms; so it suffices to compute
the action of the composition in question on decomposable tensors
in $R\ot_RM$.

 Once again, given any elements $r\in R$ and $m\in M$, let us
choose an element $f\in R$ such that $fr=r$ in~$R$.
 Then $frm=m$ in~$M$.
 Finally, we have $r\ot m\longmapsto rm\longmapsto f\ot rm=
fr\ot m=r\ot m$ in $R\ot_RM$.
\end{proof}

\begin{cor} \label{s-unital-implies-t-unital-for-rings}
 Any left s\+unital ring is t\+unital.
 Similarly, any right s\+unital ring is t\+unital.
\end{cor}

\begin{proof}
 Follows from Theorem~\ref{main-s-unital-tensor-product-theorem}
and the definition of a t\+unital ring.
\end{proof}

 In Corollary~\ref{left-s-unital-right-flat} below, we will show
that, for any left s\+unital ring $R$, the unital right
$\widetilde R$\+module $R$ is flat.

 The next lemma is a version of
Lemma~\ref{t-unital-tensor-product-lemma} for the s\+unitality property.

\begin{lem} \label{s-unital-tensor-product-lemma}
 Let $R'$ and $R''$ be arbitrary (nonunital) rings. \par
\textup{(a)} Let $B$ be an $R'$\+$R''$\+bimodule and $M$ be
an left $R''$\+module.
 Assume that $B$ is s\+unital as a left $R'$\+module.
 Then the left $R'$\+module $B\ot_{R''}M$ is also s\+unital. \par
\textup{(b)} Let $N$ be a right $R'$\+module and $B$ be
an $R'$\+$R''$\+bimodule.
 Assume that $B$ is s\+unital as a right $R''$\+module.
 Then the right $R''$\+module $N\ot_{R'}B$ is also s\+unital.
\end{lem}

\begin{proof}
 Let us explain part~(a).
 The point is that the left $R'$\+module $B\ot_{R''}M$ is a quotient
module of an infinite direct sum of copies of the left $R'$\+module~$B$.
 Therefore, it remains to refer to
Proposition~\ref{s-unital-hereditary-torsion}.
\end{proof}

\begin{cor} \label{s-unital-iff-t-unital-for-modules}
\textup{(a)} Over a left s\+unital ring, a left module is t\+unital
if and only if it is s\+unital. \par
\textup{(b)} Over a right s\+unital ring, a right module is t\+unital
if and only if it is s\+unital.
\end{cor}

\begin{proof}
 Let us prove part~(a).
 Let $R$ be a left s\+unital ring.
 If a left $R$\+module $M$ is s\+unital, then $M$ is t\+unital by
Theorem~\ref{main-s-unital-tensor-product-theorem}(a).
 Conversely, for any left $R$\+module $M$, the left $R$\+module
$R\ot_RM$ is s\+unital by Lemma~\ref{s-unital-tensor-product-lemma}(a).
 Hence, if $M$ is t\+unital, then $M$ is s\+unital.
\end{proof}

\begin{cor} \label{s-unital-monoidal-module-categories}
 Let $R$ be an s\+unital ring.
 Then the abelian category of s\+unital $R$\+$R$\+bimodules
$R\sBimods R$ is an associative and unital monoidal category, with
the unit object $R\in R\sBimods R$, with respect to the tensor
product operation\/~$\ot_R$.
 The abelian category of s\+unital left $R$\+modules $R\Modls$ is
an associative and unital left module category over $R\sBimods R$,
and the abelian category of s\+unital right $R$\+modules $\Modrs R$
is an associative and unital right module category over $R\sBimods R$,
with respect to the tensor product operation\/~$\ot_R$.
\end{cor}

\begin{proof}
 The three categories are abelian
by Proposition~\ref{s-unital-hereditary-torsion}.
 The assertions of the corollary follow from
Corollaries~\ref{t-unital-monoidal-module-categories},
\ref{s-unital-implies-t-unital-for-rings},
and~\ref{s-unital-iff-t-unital-for-modules}.
\end{proof}

\begin{ex} \label{t-unital-not-s-unital-easy-counterex}
 Here is an example of a finitely generated free abelian group $R$ with
a structure of associative ring which is t\+unital but \emph{neither}
left \emph{nor} right s\+unital.

 (0)~We start from the observation that any homomorphic image of
a left (right) s\+unital ring is left (right) s\+unital.
 This follows immediately from the definitions.

 (1)~Another simple observation is that the direct product of any two
t\+unital rings is t\+unital.
 Let $R$ and $S$ be t\+unital rings, and let $T=R\times S$ be their
direct product (with componentwise addition and multiplication).
 Let us show that the ring $T$ is t\+unital.
 It is clear that $R\ot_TR=R\ot_RR=R$ and $S\ot_TS=S\ot_SS=S$.
 It remains to check that $R\ot_TS=0$ and $S\ot_TR=0$.
 Here we notice that any t\+unital ring is idempotent, i.~e., $R^2=R$.
 Now we have $r'r''\ot_Ts=r'\ot_Tr''s=0$ for any $r'$, $r''\in R$
and $s\in S$, and similarly $s\ot_Tr'r''=0$.

 (2)~The following example of a left unital, but not right s\+unital
associative ring $D$ can be found in~\cite[Solution to
Exercise~1.10(b)]{Lam} and~\cite[Example~2.6 of the published version
or Example~6 of the \texttt{arXiv} version]{Nys}.
 Consider the matrix ring $D=
\left(\begin{smallmatrix}\boZ & \boZ \\ 0 & 0\end{smallmatrix}\right)$.
 Explicitly, $D\simeq\boZ\oplus\boZ$ is the free abelian group with
two generators endowed with the multiplication $(a,b)(c,d)=(ac,ad)$
for all $a$, $b$, $c$, $d\in\boZ$.
 Then the ring $D$ has a left unit $(1,0)$, but is not right
s\+unital, as $(0,1)D=0$ in~$D$.

 (3)~The ring $D$ is left s\+unital, so it is t\+unital by
Corollary~\ref{s-unital-implies-t-unital-for-rings}.
 Similarly, the opposite ring $D^\rop=
\left(\begin{smallmatrix}0 & \boZ \\ 0 & \boZ \end{smallmatrix}\right)$
is right s\+unital (hence t\+unital), but not left s\+unital.
 Now the ring $R=D\times D^\rop$ is t\+unital by~(1).
 But both the rings $D$ and $D^\rop$ are homomorphic images of $R$,
so $R$ is neither left nor right s\+unital by~(0).
\end{ex}

 An example of \emph{commutative} associative ring $R$ which is
t\+unital but not s\+unital will be given below in
Example~\ref{commutative-t-unital-not-s-unital}.

\Section{c-Unital Modules}  \label{c-unital-secn}

 Let $R$, $R'$, $R''$ be (nonunital) rings.
 Given two (nonunital) left $R$\+modules $L$ and $M$, we denote by
$\Hom_R(L,M)$ the abelian group of morphisms $L\rarrow M$ in
the category $R\Modln\simeq\widetilde R\Modl$.
 So there is a natural (identity) isomorphism of abelian groups
\begin{equation} \label{nonunital-and-unital-Hom}
  \Hom_R(L,M)\simeq\Hom_{\widetilde R}(L,M).
\end{equation}
 For any left $R''$\+module $M$, any $R'$\+$R''$\+bimodule $B$,
and any left $R'$\+module $P$, there is a natural isomorphism of
abelian groups
\begin{equation} \label{nonunital-two-rings-Hom-associativity}
 \Hom_{R'}(B\ot_{R''}M,\>P)\simeq\Hom_{R''}(M,\Hom_{R'}(B,P)).
\end{equation}
 The adjunction/associativity
isomorphism~\eqref{nonunital-two-rings-Hom-associativity}
can be obtained from the similar isomorphism for unital rings
$\widetilde R'$ and $\widetilde R''$ and unital (bi)modules over them
using the category
equivalences~\eqref{nonunital-mods-as-mods-over-unitalization} and
the isomorphisms~\eqref{nonunital-and-unital-tensor-product}
and~\eqref{nonunital-and-unital-Hom}.

 For any left $R$\+module $P$, there is a natural left $R$\+module
morphism
\begin{equation} \label{Hom-unitality-map}
 p\longmapsto(r\mapsto rp): P\rarrow\Hom_R(R,P).
\end{equation}

\begin{defn}
 We will say that a left $R$\+module $P$ is \emph{c\+unital} if
the natural map $P\rarrow\Hom_R(R,P)$ \,\eqref{Hom-unitality-map}
is an isomorphism.
 (The letter~``c'' here means ``contra'' or ``contramodule''.
 What we call ``c\+unital modules'' are called ``closed modules''
in~\cite[Definition~5.2]{Quil}.)
 
\end{defn}

\begin{lem} \label{t-unital-c-unital-Hom-lemma}
 Let $B$ be an $R'$\+$R''$\+bimodule and $P$ be a left $R'$\+module.
 Assume that $B$ is t\+unital as a right $R''$\+module.
 Then the left $R''$\+module\/ $\Hom_{R'}(B,P)$ is c\+unital.
\end{lem}

\begin{proof}
 This is the Hom version of Lemma~\ref{t-unital-tensor-product-lemma}.
 The assertion follows from the associativity
isomorphism~\eqref{nonunital-two-rings-Hom-associativity} and
commutativity of the suitable triangular diagram formed by
the associativity isomorphism together with (the maps induced by)
the unitality comparison maps~\eqref{tensor-right-unitality-map}
and~\eqref{Hom-unitality-map}.
\end{proof}

\begin{cor} \label{c-unitality-created-by-Hom-from-R}
 Let $R$ be a t\+unital ring.
 Then, for any left $R$\+module $P$, the left $R$\+module
$\Hom_R(R,P)$ is c\+unital.
\end{cor}

\begin{proof}
 This is the Hom version of
Corollary~\ref{t-unitality-created-by-tensor-with-R}.
 To prove the assertion, take $B=R=R'=R''$
in Lemma~\ref{t-unital-c-unital-Hom-lemma}.
\end{proof}

 Given a ring $R$, we will denote by $R\Modlc$ the full subcategory
of c\+unital $R$\+modules in $R\Modln$.

\begin{cor} \label{c-unital-module-category}
 Let $R$ be a t\+unital ring.
 Then the opposite category $R\Modlc^\sop$ to the additive category
of c\+unital left $R$\+modules $R\Modlc$ is an associative and unital
right module category over the monoidal category $R\tBimodt R$,
with respect to the Hom operation
$$
 P^\sop*B=\Hom_R(B,P)^\sop
 \quad\text{for all $P\in R\Modlc$ and $B\in R\tBimodt R$}.
$$
\end{cor}

\begin{proof}
 Follows from the associativity
isomorphism~\eqref{nonunital-two-rings-Hom-associativity}
and Lemma~\ref{t-unital-c-unital-Hom-lemma}.
\end{proof}

\begin{lem}  \label{c-unital-modules-closure-properties}
 Let $R$ be an arbitrary (nonunital) ring.
 Then the full subcategory of c\+unital left $R$\+modules
$R\Modlc\subset R\Modln$ is closed under extensions and all limits
(including kernels, direct products, and inverse limits) in
the abelian category of nonunital modules $R\Modln$.
\end{lem}

\begin{proof}
 This is our version of~\cite[Proposition~5.5]{Quil}.
 Notice that it is \emph{not} claimed in the lemma that $R\Modlc$ is
closed under submodules in $R\Modln$, but only under kernels.
 The proof is dual-analogous to that of
Lemma~\ref{t-unital-modules-closure-properties}.
 The closedness under limits follows from the fact that the Hom functor
$\Hom_R(R,{-})\simeq\Hom_{\widetilde R}(R,{-})$ preserves limits.

 To prove the closedness under extensions, consider a short exact
sequence $0\rarrow K\rarrow L\rarrow M\rarrow0$ in $R\Modln$ with
$K$, $M\in R\Modlc$.
 The assertion follows from commutativity of the diagram
$$
 \xymatrix{
  0 \ar[r] & K \ar[r] \ar@{=}[d] & L \ar[r] \ar[d]
  & M \ar[r] \ar@{=}[d] & 0 \\
  0 \ar[r] & \Hom_R(R,K) \ar[r] & \Hom_R(R,L) \ar[r] & \Hom_R(R,M)
 }
$$
with a short exact sequence in the upper row and a left exact sequence
in the lower row.
\end{proof}

\begin{rem} \label{Ext-perpendicularity-remark}
 Similarly to Remark~\ref{Tor-perpendicularity-remark}, the construction
of the full subcategory $R\Modlc\subset R\Modln$ can be viewed as
a particular case of the construction of
the \emph{$\Ext^{0,1}$\+perpendicular subcategory}
from~\cite[Proposition~1.1]{GL} or~\cite[Theorem~1.2(a)]{Pcta}.
 Given a nonunital left $R$\+module $P$, the natural map
$P=\Hom_{\widetilde R}(\widetilde R,P)\rarrow
\Hom_{\widetilde R}(R,P)=\Hom_R(R,P)$ is an isomorphism
\emph{if and only if} one has $\Hom_{\widetilde R}(\widetilde R/R,P)
\allowbreak=0=\Ext^1_{\widetilde R}(\widetilde R/R,P)$.
 (This observation is a particular case of the discussion
in~\cite[Definition~5.2]{Quil}.)
 Therefore, Lemma~\ref{c-unital-modules-closure-properties} becomes
a particular case of~\cite[first paragraph of Proposition~1.1]{GL}.

 Furthermore, the full subcategory $R\Modlc$ is closed under cokernels
in $R\Modln$ (and, consequently, abelian) whenever the projective
dimension of the unital left $\widetilde R$\+module $\widetilde R/R$
does not exceed~$1$, i.~e., in other words, the unital
$\widetilde R$\+module $R$ is projective.
 This can be also seen directly from the definition: if the functor
$\Hom_R(R,{-})\:R\Modln\rarrow R\Modln$ is exact (preserves cokernels),
then the full subcategory $R\Modlc\subset R\Modln$ is closed under
cokernels.
 The following example shows that this does \emph{not} hold in general.
\end{rem}

\begin{ex} \label{two-variables-non-c-unital-cokernel-counterex}
 Continuing the discussion from
Example~\ref{two-variables-non-t-unital-kernel-counterex}, let $R$ be
the ideal $R=(x,y)\subset\widetilde R=\boZ[x,y]$.
 Then the projective dimension of the unital $\widetilde R$\+module
$\boZ=\widetilde R/R$ is equal to~$2$, and the projective dimension of
the unital $\widetilde R$\+module $R$ is equal to~$1$.
 We claim that the full subcategory of c\+unital $R$\+modules
$R\Modlc$ is \emph{not} closed under cokernels in $R\Modln=
\widetilde R\Modl$, \emph{not} even under the cokernels
of monomorphisms.

 Indeed, consider the free $\widetilde R$\+module with one generator
$F=\boZ[x,y]$ and its quotient $\widetilde R$\+module
$G=F/xF=\boZ[y]$.
 Then one can easily see that $F\in R\Modlc$ but $G\notin R\Modlc$.
 Indeed, the $R$\+module map $r\longmapsto\frac ry \bmod \frac xy\:
R\rarrow G$ does not arise from any element of~$G$.

 The ring $R$ is this example is \emph{not} t\+unital; in fact,
$R^2\ne R$.
 Examples of t\+unital rings $T$ for which the full subcategory
$T\Modlc$ is not closed under cokernels in $T\Modln$ will be given
below in Examples~\ref{one-variable-non-c-unital-cokernel-counterex}
and~\ref{t-unital-ring-non-c-unital-cokernel-counterex}.
 It would be interesting to find a counterexample of an s\+unital
ring $S$ for which the full subcategory $S\Modlc$ is not closed
under cokernels in $S\Modln$.
\end{ex}

 For any t\+unital ring $R$,
Theorem~\ref{main-two-category-equivalences-thm}(b) below tells that
the category $R\Modlc$ is abelian.
 On the other hand, again over a t\+unital ring~$R$, according to
Corollary~\ref{closed-under-cokernels-iff-projective}, the full
subcategory of c\+unital modules $R\Modlc$ is closed under cokernels
in $R\Modln$ \emph{if and only if} $R$ is a projective unital left
$\widetilde R$\+module.

 In the more restrictive setting of \emph{rings with enough
idempotents} considered in Section~\ref{rings-from-categories-secn}
below, we will show that both the full subcategories $R\Modlt$ and
$R\Modlc$ are closed under both the kernels and cokernels in $R\Modln$.

\Section{Counterexamples}

 In this section we provide some more advanced versions of
counterexamples from
Examples~\ref{two-variables-non-t-unital-kernel-counterex},
\ref{t-unital-not-s-unital-easy-counterex},
and~\ref{two-variables-non-c-unital-cokernel-counterex}.
 All the examples in this section are commutative, associative rings.
 We start with an example of commutative t\+unital but not
s\+unital ring, continue with an example of idempotent but not
t\+unital ring, and finally come to an example of t\+unital
ring for which the class of t\+unital modules is not closed
under kernels and the class of c\+unital modules is not closed
under cokernels. 
 We also provide an example of a t\+unital ring for which
$R\Modlt$ is closed under kernels, but $R\Modlc$ is not closed
under cokernels in $R\Modln$.

\begin{ex} \label{commutative-t-unital-not-s-unital}
 Here is an example of a \emph{commutative} associative ring $R$ which
is t\+unital but \emph{not} s\+unital (see
Example~\ref{t-unital-not-s-unital-easy-counterex} for an easier
noncommutative example).
 Let $\widetilde R=\boZ[z^q\mid q\in\boQ_{\ge0}]$ be the free abelian
group spanned by the formal expressions~$z^q$, where $z$~is a variable
and $q$~ranges over the nonnegative rational numbers.
 The multiplication given by the obvious rule $z^{q'}z^{q''}=z^{q'+q''}$
makes $\widetilde R$ an associative, commutative, unital ring.
 Let $R\subset\widetilde R$ be the subgroup spanned by~$z^q$ with
$q>0$.
 Then $R$ is an ideal in $\widetilde R$, and $\widetilde R=\boZ\oplus R$
is the unitalization of the nonunital ring~$R$.
 We claim that $R$ is a t\+unital, but not s\+unital ring.

 Indeed, for any nonzero element $r\in R$ one has $r\notin Rr$, so
$R$ is not s\+unital.
 In other words, the ring $R$ is not s\+unital, since $R\ne0$ and
the ring $\widetilde R$ has no zero-divisors.
 On the other hand, the multiplication map $R\ot_RR\rarrow R$ is
surjective, since one has $z^{q/2}\ot z^{q/2}\longmapsto z^q$ for
every positive rational number~$q$.
 To show that the multiplication map is also injective, consider
any quadruple of positive rational numbers $a$, $b$, $c$, $d$ such
that $a+b=c+d$ in~$\boQ$.
 We have to show that $z^a\ot z^b=z^c\ot z^d$ in $R\ot_RR$.
 Indeed, if $a=c$ and $b=d$, then there is nothing to prove.
 Otherwise, assume without loss of generality that $a>c$ and $b<d$.
 Then we have $z^a\ot z^b=z^cz^{a-c}\ot z^b=z^c\ot z^{a-c}z^b=
z^c\ot z^d$ in $R\ot_RR$, as $a-c+b=d$ in~$\boQ$.
\end{ex}

\begin{ex} \label{one-variable-non-c-unital-cokernel-counterex}
 Let $R=\boZ[z^q\mid q\in\boQ_{>0}]$ be the ring from
Example~\ref{commutative-t-unital-not-s-unital}.
 Notice that $R=\bigcup_{n\ge1}z^{1/n}\widetilde R$ is a flat
unital $\widetilde R$\+module.
 According to Remark~\ref{Tor-perpendicularity-remark}, it follows
that the full subcategory of t\+unital $R$\+modules $R\Modlt$ is
closed under kernels in $R\Modln$.
 So $R\Modlt$ is an abelian category with an exact fully faithful
inclusion functor $R\Modlt\rarrow R\Modln$ (even though the ring $R$
is not s\+unital).
 
 On the other hand, the unital $\widetilde R$\+module $R$ is \emph{not}
projective.
 Indeed, the telescope construction provides a short exact sequence of
unital $\widetilde R$\+modules $0\rarrow
\bigoplus_{n\ge1}z^{1/n}\widetilde R\rarrow
\bigoplus_{n\ge1}z^{1/n}\widetilde R\rarrow R\rarrow0$.
 If $R$ were a projective unital $\widetilde R$\+module, it would follow
that this short exact seqeunce splits; so there is a section
$s\:R\rarrow\bigoplus_{n\ge1}z^{1/n}\widetilde R$.
 The image of the element $z^1\in R$ would have to belong to a finite
subsum $\bigoplus_{i=1}^m z^{1/n_i}\widetilde R$ of the direct sum in
the right-hand side.
 Since $z^{1/n}\widetilde R$ are torsion-free $\widetilde R$\+modules,
it would follow the image of the whole map~$s$ is contained in
the same finite subsum $\bigoplus_{i=1}^m z^{1/n_i}\widetilde R\subset
\bigoplus_{n\ge1}z^{1/n}\widetilde R$; a contradiction.
{\hbadness=1850\par}

 According to Corollary~\ref{closed-under-cokernels-iff-projective}
below, we can conclude that the full subcategory $R\Modlc$ is
\emph{not} closed under cokernels in $R\Modln$.
 (Still, the category $R\Modlc$ is abelian by
Theorem~\ref{main-two-category-equivalences-thm}(b).)
 A more explicit example of such behavior, with a discussion not
relying on references to later material, will be given in
Example~\ref{t-unital-ring-non-c-unital-cokernel-counterex}.
\end{ex}

\begin{ex} \label{idempotent-not-t-unital}
 Here is an example of an associative, commutative ring $S$ which is
idempotent, but not t\+unital; in other words, the natural map
$S\ot_SS\rarrow S$ is surjective, but not injective.
 Let $R\subset\widetilde R=\boZ[z^q\mid q\in\boQ_{\ge0}]$ be the rings
from Example~\ref{commutative-t-unital-not-s-unital}.
 Denote by $J\subset R$ the additive subgroup spanned by~$z^q$ with
$q\ge1$.
 Then $J$ is an ideal in~$\widetilde R$.
 Put $S=R/J$ and $\widetilde S=\widetilde R/J$.
 Clearly, $\widetilde S=\boZ\oplus S$ is the unitalization of
the nonunital ring~$S$.
 Denote by $\bar z^q=z^q+J\in\widetilde S$ the images of the basis
vectors $z^q\in\widetilde R$ under the surjective ring homomorphism
$\widetilde R\rarrow\widetilde S$.

 We have $S^2=S$, since $S$ is a homomorphic image of $R$ and $R^2=R$.
 Let us show that the map $S\ot_SS\rarrow S$ is \emph{not} injective.
 For any positive rational numbers $a$ and~$b$ such that $a+b=1$ we
have $\bar z^a\ot\bar z^b=\bar z^{1/2}\ot\bar z^{1/2}$ in $S\ot_SS$.
 The image of this element under the multiplication map
$S\ot_SS\rarrow S$ is $\bar z^1=0$ in~$S$.
 Now we claim that the element $\bar z^{1/2}\ot\bar z^{1/2}$ is nonzero
in $S\ot_SS$.

 One can argue as follows.
 Endow the rings $\widetilde R$ and $\widetilde S$ with a grading
indexed by the abelian group of rational numbers, given by the rule
$\deg z^q=q=\deg\bar z^q$.
 Then there is the induced $\boQ$\+valued grading on the tensor
products $R\ot_RR$ and $S\ot_SS$.
 The nonunital rings $R$ and $S$ are concentrated in positive degrees,
and the surjective ring homomorphism $R\rarrow S$ is an isomorphism
in the degrees~$<\nobreak1$.
 It follows easily that the induced surjective map $R\ot_RR\rarrow
S\ot_SS$ is an isomorphism in the degrees~$\le\nobreak1$.
 So $z^{1/2}\ot z^{1/2}\ne0$ in $R\ot_RR$ implies $\bar z^{1/2}\ot
\bar z^{1/2}\ne0$ in $S\ot_SS$.
\end{ex}

 The following lemma is standard.

\begin{lem}[K\"unneth formula] \label{kuenneth-formula}
 Let $A$ and $B$ be associative, unital rings.
 Let $K$ be a right $A$\+module, $N$ be a right $B$\+module,
$L$ be a left $A$\+module, and $M$ be a left $B$\+module.
 Assume that all the abelian groups $A$, $B$, $K$, $N$, $L$, $M$
have no torsion; and assume further that all the abelian groups\/
$\Tor_i^A(K,L)$ and\/ $\Tor_j^B(N,M)$, \ $i$, $j\ge0$, have
no torsion, either.
 Then there is a natural isomorphism of abelian groups
$$
 \Tor_n^{A\ot_\boZ B}(K\ot_\boZ N,\>L\ot_\boZ M)
 \simeq\bigoplus\nolimits_{i+j=n}
 \Tor_i^A(K,L)\ot_\boZ\Tor_j^B(N,M)
$$
for every $n\ge0$.
\end{lem}

\begin{proof}
 Pick a flat resolution $F_\bu$ for the left $A$\+module $L$
and a flat resolution $G_\bu$ for the left $B$\+module~$M$.
 Then $F_\bu\ot_\boZ G_\bu$ is a flat resolution for the left
$(A\ot_\boZ B)$\+module $L\ot_\boZ M$, and one can compute
$(K\ot_\boZ N)\ot_{A\ot_\boZ B}(F_\bu\ot_\boZ G_\bu)=
(K\ot_AF_\bu)\ot_\boZ(N\ot_BG_\bu)$.
 Finally, for any (bounded above) complexes of torsion-free abelian
groups $C_\bu$ and $D_\bu$ with torsion-free homology groups, one has
$H_n(C_\bu\ot_\boZ D_\bu)=\bigoplus_{i+j=n}H_i(C_\bu)\ot_\boZ
H_j(D_\bu)$ for all $n\in\boZ$.
\end{proof}

 The following remark implies that t\+unital rings $R$
are defined by \emph{a low-degree part of} the set of conditions
expressed by saying that the unital ring homomorphism $\widetilde R
\rarrow\widetilde R/R=\boZ$ is a \emph{homological ring
epimorphism}~\cite[Section~4]{GL}.
 We are grateful to Jan \v St\!'ov\'\i\v cek for suggesting
this interpretation.

\begin{rem} \label{t-unital-rings-characterized-remark}
 Let $R$ be a (nonunital) associative ring.
 As a particular case of Remark~\ref{Tor-perpendicularity-remark},
we know that the ring $R$ is t\+unital if and only if
$\widetilde R/R\ot_{\widetilde R}R=0=
\Tor_1^{\widetilde R}(\widetilde R/R,R)$.
 Notice that the map $\widetilde R/R\ot_{\widetilde R}R\rarrow
{\widetilde R/R}\ot_{\widetilde R}\widetilde R=\widetilde R/R=\boZ$
induced by the inclusion $R\rarrow\widetilde R$ vanishes.
 Consequently, $\widetilde R/R\ot_{\widetilde R}R\simeq
\Tor_1^{\widetilde R}(\widetilde R/R,\widetilde R/R)$.
 Also we have $\Tor_1^{\widetilde R}(\widetilde R/R,R)\simeq
\Tor_2^{\widetilde R}(\widetilde R/R,\widetilde R/R)$.
 Thus a ring $R$ is t\+unital if and only if
$\Tor_1^{\widetilde R}(\widetilde R/R,\widetilde R/R)=0
=\Tor_2^{\widetilde R}(\widetilde R/R,\widetilde R/R)$, or in
other notation, if and only if $\Tor_1^{\widetilde R}(\boZ,\boZ)=0=
\Tor_2^{\widetilde R}(\boZ,\boZ)$.
\end{rem}

\begin{ex} \label{rational-power-polynomials-two-variables-t-unital}
 Let $\widetilde R=\boZ[x^p\mid p\in\boQ_{\ge0}]$ and
$\widetilde S=\boZ[y^q\mid q\in\boQ_{\ge0}]$ be two copies of
the unital ring from Example~\ref{commutative-t-unital-not-s-unital}.
 Consider the ring $\widetilde T=\widetilde R\ot_\boZ\widetilde S
=\boZ[x^py^q\mid p,q\in\boQ_{\ge0}]$.
 So $\widetilde T$ is an associative, commutative, unital ring.
 Let $T\subset\widetilde T$ be the subgroup spanned by $x^py^q$
with $(p,q)\ne(0,0)$.
 Then $T$ is an ideal in $\widetilde T$, and $\widetilde T=\boZ\oplus T$
is the unitalization of the nonunital ring~$T$.
 We claim that $T$ is a t\+unital ring.

 Indeed, following Remark~\ref{t-unital-rings-characterized-remark},
it suffices to check that $\Tor_1^{\widetilde T}(\boZ,\boZ)=0=
\Tor_2^{\widetilde T}(\boZ,\boZ)$.
 One easily computes that $\Tor^{\widetilde R}_i(\boZ,\boZ)=
\Tor^{\widetilde R}_{i-1}(\boZ,R)=0$ for all $i\ge1$ (in fact,
$R=\bigcup_{n\ge1}x^{1/n}\widetilde R$ is a flat $\widetilde R$\+module
and $R^2=R$).
 Similarly, $\Tor^{\widetilde S}_j(\boZ,\boZ)=0$ for all $j\ge1$.
 On the other hand, $\Tor_0^{\widetilde R}(\boZ,\boZ)=\boZ=
\Tor_0^{\widetilde S}(\boZ,\boZ)$.
 By the K\"unneth formula (Lemma~\ref{kuenneth-formula}), it follows
that $\Tor^{\widetilde T}_n(\boZ,\boZ)=0$ for all $n\ge1$.
\end{ex}

\begin{ex} \label{t-unital-ring-non-t-unital-kernel-counterex}
 This is an improved version of
Example~\ref{two-variables-non-t-unital-kernel-counterex}, providing
a commutative t\+unital ring $T$ for which the full subcategory
of t\+unital $T$\+modules $T\Modlt$ is \emph{not} closed under kernels
(\emph{not} even under the kernels of epimorphisms) in the abelian category of nonunital $T$\+modules $T\Modln$.
 We keep the notation of
Example~\ref{rational-power-polynomials-two-variables-t-unital}
for the rings $R$, $S$, and~$T$.
 
 Consider the localizations $\widetilde R[x^{-1}]=
\boZ[x^p\mid p\in\boQ]$ and $\widetilde S[y^{-1}]=
\boZ[y^q\mid q\in\boQ]$ of the rings $R$ and~$S$.
 Let $\widetilde P_x=\widetilde R[x^{-1}]/\widetilde R$ and
$\widetilde P_y=\widetilde S[y^{-1}]/\widetilde S$ be
the related generalized Pr\"ufer modules.
 Then $\Tor_0^{\widetilde R}(\boZ,\widetilde P_x)=0$ and
$\Tor_1^{\widetilde R}(\boZ,\widetilde P_x)=\boZ$.
 By the K\"unneth formula, it follows that
$\Tor_0^{\widetilde T}(\boZ,\>\widetilde P_x\ot_\boZ\widetilde P_y)=0
=\Tor_1^{\widetilde T}(\boZ,\>\widetilde P_x\ot_\boZ\widetilde P_y)$.
 Thus $\widetilde P_x\ot_\boZ\widetilde P_y$ is a t\+unital
$T$\+module (by Remark~\ref{Tor-perpendicularity-remark}).
 
 On the other hand, denote by $\widetilde Q_y=
\boZ\ot_\boZ \widetilde P_y$ the $\widetilde S$\+module
$\widetilde P_y$ viewed as a $\widetilde T$\+module
with the given (Pr\"ufer) action of the elements~$y^q$ and
the zero action of the elements $x^py^q\in T$ for $p>0$.
 Then, again by the K\"unneth formula, $\Tor_1^{\widetilde T}
(\boZ,\widetilde Q_y)=\Tor_1^{\widetilde S}(\boZ,\widetilde P_y)=\boZ$.
 So Remark~\ref{Tor-perpendicularity-remark} tells that
$\widetilde Q_y\notin T\Modlt$.

 Consider another version of generalized Pr\"ufer module
over~$\widetilde R$: namely, put $P_x=\widetilde R[x^{-1}]/R$.
 Then $P_x=\varinjlim_{n\ge1}\widetilde R[x^{-1}]/x^{1/n}\widetilde R
\simeq\varinjlim_{n\ge1}\widetilde P_x$ is a direct limit of copies
of the $\widetilde R$\+module $\widetilde P_x$.
 Furthermore, we have a short exact sequence of $\widetilde R$\+modules
$0\rarrow\boZ\rarrow P_x\rarrow\widetilde P_x\rarrow0$.
 The $\widetilde T$\+module $P_x\ot_\boZ\widetilde P_y$ is a direct
limit of copies of the $\widetilde T$\+module $\widetilde P_x\ot_\boZ
\widetilde P_y$; so Lemma~\ref{t-unital-modules-closure-properties}
tells that $P_x\ot_\boZ\widetilde P_y$ is a t\+unital T\+module.
 Finally, in the short exact sequence of $T$\+modules
$0\rarrow\widetilde Q_y\rarrow P_x\ot_\boZ\widetilde P_y\rarrow
\widetilde P_x\ot_\boZ\widetilde P_y\rarrow0$ the middle term and
the cokernel are t\+unital, but the kernel is not.
\end{ex}

\begin{ex} \label{t-unital-ring-non-c-unital-cokernel-counterex}
 This is an improved version of
Example~\ref{two-variables-non-c-unital-cokernel-counterex},
providing a commutative t\+unital ring $T$ for which the full
subcategory of c\+unital $T$\+modules $T\Modlc$ is \emph{not} closed
under cokernels (\emph{not} even under the cokernels of monomorphisms)
in the abelian category of nonunital $T$\+modules $T\Modln$.
 We still keep the notation of
Example~\ref{rational-power-polynomials-two-variables-t-unital}
for the rings $R$, $S$, and~$T$.

 Consider the flat $\widetilde T$\+module $\widetilde F=
\widetilde R\ot_\boZ S\subset T$ and its flat
$\widetilde T$\+submodule $F=R\ot_\boZ S\subset\widetilde F$.
 So the free abelian group $\widetilde F$ is spanned by
the elements~$x^py^q$ with $p\ge0$, \,$q>0$, while the free abelian
group $F$ is spanned by the elements~$x^py^q$ with $p>0$, \,$q>0$.
 Then both the $T$\+modules $F$ and $\widetilde F$ are c\+unital.
 In other words, this means that any $\widetilde T$\+module morphism
$T\rarrow F$ can be uniquely extended to a $\widetilde T$\+module
morphism $\widetilde T\rarrow F$, and similarly for morphisms
into~$\widetilde F$.

 Indeed, let $h\:T\rarrow F$ be a $T$\+module morphism.
 Consider the element $h(xy)\in F$.
 Then we have $h(xy)=xh(y)=yh(x)$ in $F$, and so the element
$h(xy)$ is divisible \emph{both} by~$x$ and by~$y$ in~$F$.
 It follows easily from the construction of $F$ that the element
$h(xy)$ must be divisible by~$xy$ in $F$, i.~e., there exists
an element $f\in F$ such that $h(xy)=xyf$.
 Hence the morphism $h\:T\rarrow F$ comes from the element $f\in F$
via the natural map $F\rarrow\Hom_T(T,F)$.
 The argument for $\widetilde F$ is exactly the same.
 Arguing in this way, one can show that any flat unital
$\widetilde T$\+module is a c\+unital $T$\+module.

 On the other hand, notice that the $S$\+module $S$ is \emph{not}
c\+unital, as the identity morphism $\id_S\:S\rarrow S$ does not belong
to the image of the natural map $S\rarrow\Hom_S(S,S)$.
 In other words, the morphism~$\id_S$ cannot be extended to
an $S$\+module morphism $\widetilde S\rarrow S$.
 Essentially for this reason, the quotient $T$\+module
$\widetilde F/F=S$ (with the zero action of the elements $x^py^q\in T$
for $p>0$) is also \emph{not} c\+unital.
 Indeed, the surjective morphism $T\rarrow T/(R\ot_\boZ\widetilde S)=S$
does not belong to the image of the natural map $S\rarrow\Hom_T(T,S)$.
 So, in the short exact sequence of $T$\+modules $0\rarrow F\rarrow
\widetilde F\rarrow S\rarrow0$ the middle term and the kernel are
c\+unital, but the cokernel is not c\+unital.

 Notice that the ring $T$ in this example is t\+unital, but
\emph{not} s\+unital.
 It would be interesting to find an example of s\+unital ring for which
the full subcategory of c\+unital modules is not closed under cokernels
in the ambient abelian category of nonunital modules.
\end{ex}

\Section{The Main Abelian Category Equivalence}
\label{equivalence-secn}

 Let $R$ be a (nonunital) ring.
 We will say that a left $R$\+module $M$ is a \emph{null-module} if
the action of $R$ in $M$ is zero, i.~e., $rm=0$ in $M$ for all
$r\in R$ and $m\in M$.
 Right null-modules are defined similarly.
 We will denote the full subcategories of null-modules by
$R\Modlz\subset R\Modln$ and $\Modrz R\subset\Modrn R$.
 (Cf.~\cite[Definitions~2.1 and~7.1]{Quil},
\cite[Section~2.1.3]{GR}.)

 Notice that the full subcategory of null-modules $R\Modlz$ is closed
under submodules, quotients, infinite direct sums, and infinite products
(hence under all limits and colimits) in $R\Modln$.
 The following easy lemma says more.

\begin{lem}
 Let $R$ be an idempotent ring, i.~e., $R^2=R$.
 Then the full subcategory $R\Modlz$ is closed under extensions in
$R\Modln$.
\end{lem}

\begin{proof}
 Let $0\rarrow K\rarrow L\rarrow M\rarrow0$ be a short exact sequence
in $R\Modln$ with $K$, $M\in R\Modlz$.
 Given an element $l\in L$, it suffices to show that $r'r''l=0$ in $L$
for all $r'$, $r''\in R$.
 Now we have $r''(l+K)=0$ in $M$, since $M$ is a null-module.
 Hence $r''l\in K\subset L$.
 Finally $r'(r''l)=0$ in~$K$, since $K$ is a null-module.
\end{proof}

 Thus, over an idempotent ring $R$, the full subcategory of null-modules
is a thick subcategory, or in a different terminology, a \emph{Serre
subcategory} in the abelian category $R\Modln$
\,\cite[Section~1.11]{GrToh}, \cite[Section~III.1]{Gab}.
 Therefore, the abelian Serre quotient category $R\Modln/R\Modlz$
is defined.
 Ignoring set-theoretical issues, one can say that the quotient category
$R\Modln/R\Modlz$ can be obtained from $R\Modln$ by inverting all
the morphisms whose kernel and cokernel belong to $R\Modlz$.
 (See~\cite{GrToh,Gab} for a more concrete construction.)

 The aim of this section is to explain that, for a t\+unital ring $R$,
both the categories of t\+unital and c\+unital left $R$\+modules are
actually abelian, and naturally equivalent to each other and to
the Serre quotient category $R\Modln/R\Modlz$,
$$
 R\Modlt\simeq R\Modln/R\Modlz\simeq R\Modlc.
$$
 A more general version of this result is due to
Quillen~\cite[Theorems~4.5 and~5.6]{Quil}.

 The following lemma can be compared
to~\cite[Proposition~3.1(b)]{BP2} (cf.\ the paragraph preceding
Remark~\ref{t-unital-rings-characterized-remark} above).
 What we call t\+unital modules are called ``comodules'' in
the context of~\cite{BP2}.

\begin{lem} \label{t-unitalization-lemma}
 Let $R$ be a t\+unital ring.
 Then the functor $R\ot_R{-}\,\:R\Modln\rarrow R\Modlt$ is right adjoint
to the fully faithful inclusion functor $R\Modlt\rarrow R\Modln$.
 For any left $R$\+module $M$, the natural map $R\ot_RM\rarrow M$ is
the adjunction counit.
\end{lem}

\begin{proof}
 For any left $R$\+module $M$, the left $R$\+module $R\ot_RM$ is
t\+unital by Corollary~\ref{t-unitality-created-by-tensor-with-R}(a).
 Now let $L$ be a t\+unital left $R$\+module.
 We need to construct a natural isomorphism of abelian groups
$\Hom_R(L,M)\simeq\Hom_R(L,\>R\ot_RM)$.
 For this purpose, to any $R$\+module morphism $f\:L\rarrow M$
we assign the morphism $R\ot_Rf\:L\simeq R\ot_RL\rarrow R\ot_RM$.
 To any $R$\+module morphism $g\:L\rarrow R\ot_RM$ we assign
the composition $L\overset g\rarrow R\ot_RM\rarrow M$.

 Let us check that these two maps between the two groups of morphisms
are inverse to each other.
 Starting with a morphism $f\:L\rarrow M$, it is clear that
the composition $L\simeq R\ot_R L\rarrow R\ot_RM\rarrow M$
is equal to~$f$.
 Starting with a morphism $g\:L\rarrow R\ot_RM$, producing
the related morphism $f\:L\rarrow M$, and coming back to a new morphism
$g'\:L\rarrow R\ot_RM$ via the constructions above, in order to check
that $g'=g$ one needs to observe that the two multiplication maps
$R\ot_RR\ot_RM\rightrightarrows R\ot_RM$ are equal to each other.
 The latter property holds by the definition of the tensor product
of $R$\+modules $R\ot_RM$.
\end{proof}

 The next lemma is our version of~\cite[Proposition~3.2(b)]{BP2}.
 What we call c\+unital modules are called ``contramodules''
in the context of~\cite{BP2}.

\begin{lem} \label{c-unitalization-lemma}
 Let $R$ be a t\+unital ring.
 Then the functor\/ $\Hom_R(R,{-})\:R\Modln\rarrow R\Modlc$ is left
adjoint to the fully faithful inclusion functor
$R\Modlc\rarrow R\Modln$.
 For any left $R$\+module $P$, the natural map $P\rarrow\Hom_R(R,P)$ is
the adjunction unit.
\end{lem}

\begin{proof}
 For any left $R$\+module $P$, the left $R$\+module $\Hom_R(R,P)$ is
t\+unital by Corollary~\ref{c-unitality-created-by-Hom-from-R}.
 Now let $Q$ be a c\+unital left $R$\+module.
 We need to construct a natural isomorphism of abelian groups
$\Hom_R(P,Q)\simeq\Hom_R(\Hom_R(R,P),Q)$.
 For this purpose, to any $R$\+module morphism $f\:P\rarrow Q$
we assign the morphism $\Hom_R(R,f)\:\Hom_R(R,P)\rarrow\Hom_R(R,Q)
\simeq Q$.
 To any $R$\+module morphism $g\:\Hom_R(R,P)\rarrow Q$ we assign
the composition $P\rarrow\Hom_R(R,P)\rarrow Q$.

 Let us check that these two maps between the two groups of morphisms
are inverse to each other.
 Starting with a morphism $f\:P\rarrow Q$, it is clear that
the composition $P\rarrow\Hom_R(R,P)\rarrow\Hom_R(R,Q)\simeq Q$
is equal to~$f$.
 Starting with a morphism $g\:\Hom_R(R,P)\rarrow Q$, producing
the related morphism $f\:P\rarrow Q$, and coming back to a new morphism
$g'\:\Hom_R(R,P)\rarrow Q$ via the constructions above, in order to
check that $g'=g$ one needs to observe that the two maps
$\Hom_R(R,P)\rightrightarrows\Hom_R(R\ot_RR,\>P)\simeq
\Hom_R(R,\Hom_R(R,P))$ induced by the maps $R\ot_RR\rarrow R$ and
$P\rarrow\Hom_R(R,P)$ are equal to each other.
 The latter property holds by the definition of the group of
$R$\+module morphisms $\Hom_R(R,P)$.
\end{proof}

 The following lemma is due to Quillen~\cite[discussions of
formulas~(2.2) and~(5.1)]{Quil}.

\begin{lem}[\cite{Quil}]
\label{unitality-comparisons-have-null-co-kernels}
 Let $R$ be an arbitrary (nonunital) ring.  Then \par
\textup{(a)} for any left $R$\+module $M$, the kernel and cokernel of
the natural $R$\+module morphism $R\ot_RM\rarrow M$ are null-modules;
\par
\textup{(b)} for any left $R$\+module $P$, the kernel and cokernel of
the natural $R$\+module morphism $P\rarrow\Hom_R(R,P)$ are null-modules.
\end{lem}

\begin{proof}
 Part~(a): the cokernel of the map $R\ot_RM\rarrow M$ is obviously
the maximal null quotient module of~$M$.
 Concerning the kernel, let $\sum_{i=1}^nr_i\ot m_i\in R\ot_RM$ be
an element annihilated by the map $R\ot_RM\rarrow M$.
 Then, for any $r\in R$, one has $r\sum_{i=1}^nr_i\ot m_i=\sum_{i=1}^n
rr_i\ot m_i=\sum_{i=1}^n r\ot r_im_i=r\ot\sum_{i=1}^nr_im_i=0$ in
$R\ot_RM$.
 Part~(b): the kernel of the map $P\rarrow\Hom_R(R,P)$ is obviously
the maximal null submodule of~$M$.
 Concerning the cokernel, let $f\in\Hom_R(R,P)$ be a morphism and
$r\in R$ be an element.
 Then the $R$\+module morphism $rf\:R\rarrow P$, given by the rule
$(rf)(r')=f(r'r)=r'f(r)$ for all $r'\in R$, comes from the element
$f(r)\in P$ via the natural map $P\rarrow\Hom_R(R,P)$.
 So $R$ acts by zero on the coset of the element~$f$ in the cokernel
of the map $P\rarrow\Hom_R(R,P)$.
\end{proof}

 In the proof of part~(b) of the next lemma, the following well-known
fact is used.
 Let $A$ be a unital ring, $E$ be a right $A$\+module, and $M$ be
a left $A$\+module.
 Then, for every $n\ge0$, there is a natural isomorphism of abelian
groups
$$
 \Ext^n_A(M,\Hom_\boZ(E,\boQ/\boZ))\simeq
 \Hom_\boZ(\Tor_n^A(E,M),\boQ/\boZ).
$$

\begin{lem} \label{null-modules-Ext-Tor-vanishing-lemma}
 Let $R$ be a t\+unital ring, $\widetilde R=\boZ\oplus R$ be its
unitalization, and $N$ be a left null-module over~$R$.
 Then \par
\textup{(a)} $R\ot_RN=R\ot_{\widetilde R}N=
0=\Tor^{\widetilde R}_1(R,N)$; \par
\textup{(b)} $\Hom_R(R,N)=\Hom_{\widetilde R}(R,N)=0=
\Ext_{\widetilde R}^1(R,N)$.
\end{lem}

\begin{proof}
 Part~(a): one has $R\ot_RN=0$ for any null-module $N$ over
an idempotent ring~$R$, as $r'r''\ot n=r'\ot r''n=0$ in $R\ot_RN$
for any $r'$, $r''\in R$ and $n\in N$.
 To prove the $\Tor_1$ vanishing, notice that the category of (left)
null-modules over $R$ is isomorphic to the category of abelian groups:
any abelian group can be uniquely endowed with a null-module structure.
 Consider a short exact sequence of abelian groups $0\rarrow G\rarrow F
\rarrow N\rarrow0$ with a free abelian group $F$, and view it as
a short exact sequence of left null-modules over~$R$.
 We have already seen that $R\ot_{\widetilde R}G=0$.
 Following Remark~\ref{Tor-perpendicularity-remark}, we have
$\Tor^{\widetilde R}_1(R,\boZ)=\Tor^{\widetilde R}_1(R,\widetilde R/R)
=0$, since $R$ is a t\+unital right $R$\+module.
 Hence $\Tor^{\widetilde R}_1(R,F)=0$, and we can conclude that
$\Tor^{\widetilde R}_1(R,N)=0$.

 Part~(b): one has $\Hom_R(R,N)=0$ for any null-module $N$ over
an idempotent ring $R$, as $f(r'r'')=r'f(r'')=0$ in $N$ for any
$r'$, $r''\in R$ and $f\in\Hom_R(R,N)$.
 To prove the $\Ext^1$ vanishing, consider a short exact sequence of
abelian groups $0\rarrow N\rarrow J\rarrow K\rarrow0$, where $J$ is
a product of copies of $\boQ/\boZ$, and view it as a short exact
sequence of left null-modules over~$R$.
 We have already seen that $\Hom_{\widetilde R}(R,K)=0$.
 By Remark~\ref{Tor-perpendicularity-remark}, we have
$\Tor^{\widetilde R}_1(\boZ,R)=\Tor^{\widetilde R}_1(\widetilde R/R,R)
=0$, since $R$ is a t\+unital left $R$\+module.
 Hence, viewing $\boQ/\boZ$ as an abelian group \emph{and} a left
null-module ${}_{\widetilde R}\boQ/\boZ$, and viewing $\boZ$ as a right
null-module $\boZ_{\widetilde R}$, we have
$\Ext^1_{\widetilde R}(R,{}_{\widetilde R}\boQ/\boZ)=
\Ext^1_{\widetilde R}(R,\Hom_\boZ(\boZ_{\widetilde R},\boQ/\boZ))=
\Hom_\boZ(\Tor_1^{\widetilde R}(\boZ_{\widetilde R},R),\boQ/\boZ)=0$.
 Therefore, $\Ext^1_{\widetilde R}(R,J)=0$, and we can conclude that
$\Ext^1_{\widetilde R}(R,N)=0$.
\end{proof}

\begin{prop} \label{null-isomorphisms-inverted-by-tensor-Hom-R}
 Let $R$ be a t\+unital ring and $f\:L\rarrow M$ be a left $R$\+module
morphism whose kernel and cokernel are null-modules.  Then \par
\textup{(a)} the left $R$\+module morphism $R\ot_Rf\:R\ot_RL\rarrow
R\ot_RM$ is an isomorphism; \par
\textup{(b)} the left $R$\+module morphism $\Hom_R(R,f)\:\Hom_R(R,L)
\rarrow\Hom_R(R,M)$ is an isomorphism.
\end{prop}

\begin{proof}
 A more general version of part~(a) can be found
in~\cite[Proposition~4.3]{Quil}.
 Our argument is based on
Lemma~\ref{null-modules-Ext-Tor-vanishing-lemma}.
 In both parts~(a) and~(b), it suffices to consider two cases:
either~$f$ is a surjective map with a null kernel, or $f$~is
an injective map with a null cokernel.
 In the former case, we have a short exact sequence $0\rarrow N\rarrow
L\rarrow M\rarrow0$ with a null-module~$N$.
 Then the vanishing of the tensor product $R\ot_RN$ implies
an isomorphism $R\ot_RL\simeq R\ot_RM$, and
Lemma~\ref{null-modules-Ext-Tor-vanishing-lemma}(b) implies
an isomorphism $\Hom_R(R,L)\simeq\Hom_R(R,M)$.
 In the latter case, we have a short exact sequence $0\rarrow L\rarrow
M\rarrow N\rarrow0$ with a null-module~$N$.
 Then Lemma~\ref{null-modules-Ext-Tor-vanishing-lemma}(a) implies
an isomorphism $R\ot_RL\simeq R\ot_RM$, and the vanishing of the Hom
group $\Hom_R(R,N)$ implies an isomorphism
$\Hom_R(R,L)\simeq\Hom_R(R,M)$.
\end{proof}

 The following corollary is due to Quillen.
 Part~(a) is contained in~\cite[Proposition~2.6]{Quil},
part~(b) in~\cite[Proposition~4.1]{Quil}, and part~(c)
in~\cite[Proposition~5.3]{Quil}.

\begin{cor}[\cite{Quil}]
\label{null-isomorphisms-inverted-by-tensor-Hom-with-t-c}
 Let $R$ be a t\+unital ring and $f\:L\rarrow M$ be a left $R$\+module
morphism whose kernel and cokernel are null-modules.  Then \par
\textup{(a)} for any t\+unital right $R$\+module $D$, the induced map
of abelian groups $D\ot_RL\rarrow D\ot_RM$ is an isomorphism; \par
\textup{(b)} for any t\+unital left $R$\+module $E$, the induced map
of abelian groups\/ $\Hom_R(E,L)\rarrow\Hom_R(E,M)$ is an isomorphism;
\hbadness=1400\par
\textup{(c)} for any c\+unital left $R$\+module $P$, the induced map
of abelian groups\/ $\Hom_R(M,P)\rarrow\Hom_R(L,P)$ is an isomorphism.
\end{cor}

\begin{proof}
 Part~(a) follows from
Proposition~\ref{null-isomorphisms-inverted-by-tensor-Hom-R}(a)
and the natural isomorphisms $D\ot_R\nobreak K\simeq(D\ot_RR)\ot_RK
\simeq D\ot_R(R\ot_RK)$ for all left $R$\+modules~$K$
(in particular, $K=L$ or $K=M$).
 Part~(b) follows from
Proposition~\ref{null-isomorphisms-inverted-by-tensor-Hom-R}(b)
and the natural isomorphisms $\Hom_R(E,K)\simeq\Hom_R(R\ot_RE,\>K)
\simeq\Hom_R(E,\Hom_R(R,K))$ for all left $R$\+modules~$K$.
 Part~(c) follows from
Proposition~\ref{null-isomorphisms-inverted-by-tensor-Hom-R}(a)
and the natural isomorphisms $\Hom_R(K,P)\simeq\Hom_R(K,\Hom_R(R,P))
\simeq\Hom_R(R\ot_RK,\>P)$ for all left $R$\+modules~$K$.
\end{proof}

 The following theorem is a particular case of Quillen's
\cite[Theorems~4.5 and~5.6]{Quil}.

\begin{thm}[\cite{Quil}]
\label{main-two-category-equivalences-thm}
 Let $R$ be a t\+unital ring.  In this setting: \par
\textup{(a)} The additive category $R\Modlt$ is abelian, and
the functor $R\ot_R{-}\,\:R\Modln\rarrow R\Modlt$ is exact.
 This functor factorizes through the localization functor $R\Modln
\rarrow R\Modln/R\Modlz$, inducing a functor $R\Modln/R\Modlz
\rarrow R\Modlt$, which is an equivalence of abelian categories
$R\Modln/R\Modlz\simeq R\Modlt$.
 The composition of the inclusion and localization functors
$R\Modlt\rarrow R\Modln\rarrow R\Modln/R\Modlz$ provides
the inverse equivalence. \par
\textup{(b)} The additive category $R\Modlc$ is abelian, and
the functor\/ $\Hom_R(R,{-})\:R\Modln\allowbreak\rarrow R\Modlc$
is exact.
 This functor factorizes through the localization functor $R\Modln
\rarrow R\Modln/R\Modlz$, inducing a functor $R\Modln/R\Modlz
\rarrow R\Modlc$, which is an equivalence of abelian categories
$R\Modln/R\Modlz\simeq R\Modlc$.
 The composition of the inclusion and localization functors
$R\Modlc\rarrow R\Modln\rarrow R\Modln/R\Modlz$ provides
the inverse equivalence.  \hfuzz=4pt
\end{thm}

\begin{rem} \label{Giraud-remark}
 Notice that the functor $R\ot_R{-}\,\:R\Modln\rarrow R\Modln$ is
\emph{not} exact in general.
 The functor $R\ot_R{-}$ is only exact when it is viewed as
\emph{taking values in $R\Modlt$}.
 The inclusion functor $R\Modlt\rarrow R\Modln$ is also \emph{not}
exact in general, as
Example~\ref{t-unital-ring-non-t-unital-kernel-counterex}
illustrates.
 It is only the composition $R\Modlt\rarrow R\Modln\rarrow R\Modln/
R\Modlz$ that is exact.

 Similarly, the functor $\Hom_R(R,{-})\:R\Modln\rarrow R\Modln$ is
\emph{not} exact in general.
 The functor $\Hom_R(R,{-})$ is only exact when it is viewed as
\emph{taking values in $R\Modlc$}.
 The inclusion functor $R\Modlc\rarrow R\Modln$ is also \emph{not}
exact in general, as
Examples~\ref{one-variable-non-c-unital-cokernel-counterex}
and~\ref{t-unital-ring-non-c-unital-cokernel-counterex}
illustrate.
 It is only the composition $R\Modlc\rarrow R\Modln\rarrow R\Modln/
R\Modlz$ that is exact.

 The full subcategory $R\Modlc\subset R\Modln$ is a \emph{Giraud
subcategory} in the sense of~\cite[Section~X.1]{Sten}.
 Dually, the full subcategory $R\Modlt\subset R\Modln$ is
a \emph{co-Giraud subcategory}~\cite[Section~1]{CFM}.
 This terminology means that the inclusion functor $R\Modln/R\Modlz
\simeq R\Modlt\rarrow R\Modln$ is left adjoint to the localization
functor $R\Modln\rarrow R\Modln/R\Modlz$, while the inclusion
functor $R\Modln/R\Modlz\simeq R\Modlc\rarrow R\Modln$ is right
adjoint to the localization functor.
\end{rem}

 In the proof of the theorem, the following observation is used.
 Let $\sC$ and $\sD$ be two categories, and let $\sS\subset\sC$
and $\sT\subset\sD$ be two multiplicative subsets of morphisms.
 Consider the localizations $\sC[\sS^{-1}]$ and $\sD[\sT^{-1}]$
of the categories $\sC$ and~$\sD$.
 Let $F\:\sC\rarrow\sD$ and $G\:\sD\rarrow\sC$ be a pair of adjoint
functors, with the functor $F$ left adjoint to~$G$.
 Assume that the functors $F$ and $G$ descend to well-defined
functors $\overline F\:\sC[\sS^{-1}]\rarrow\sD[\sT^{-1}]$ and
$\overline G\:\sD[\sT^{-1}]\rarrow\sC[\sS^{-1}]$.
 Then the functor $\overline F$ is left adjoint to
the functor~$\overline G$.
 The adjunction unit and counit for the pair $(\overline F,
\overline G)$ are induced by the adjunction unit and counit
for the pair $(F,G)$.
 
\begin{proof}[Proof of
Theorem~\ref{main-two-category-equivalences-thm}]
 Part~(a): it is clear from
Proposition~\ref{null-isomorphisms-inverted-by-tensor-Hom-R}(a) that
the functor $R\ot_R{-}\,\:R\Modln\rarrow R\Modlt$ factorizes (uniquely)
through the localization functor $R\Modln\rarrow R\Modln/R\Modlz$,
inducing a well-defined functor from the quotient category
$R\Modln/R\Modlz\rarrow R\Modlt$.
 Then it follows from Lemma~\ref{t-unitalization-lemma} that
the resulting functor is right adjoint to the composition
$R\Modlt\rarrow R\Modln\rarrow R\Modln/R\Modlz$.
 It remains to show that the adjunction morphisms in $R\Modlt$
and $R\Modln/R\Modlz$ are isomorphisms.
 The adjunction unit $M\rarrow R\ot_RM$ (which is only defined for
$M\in R\Modlt$) is an isomorphism for every $M\in R\Modlt$
by construction.
 The adjunction counit $R\ot_RL\rarrow L$ is an isomorphism in
the quotient category $R\Modln/R\Modlz$ for every left $R$\+module $L$
by Lemma~\ref{unitality-comparisons-have-null-co-kernels}(a).
 Finally, the functor $R\ot_R{-}\,\:R\Modln\rarrow R\Modlt$ is exact
as the composition of two exact functors $R\Modln\rarrow R\Modln/
R\Modlz\rarrow R\Modlt$ (notice that any equivalence of abelian
categories is an exact functor).

 Part~(b): it is clear from
Proposition~\ref{null-isomorphisms-inverted-by-tensor-Hom-R}(b) that
the functor $\Hom_R(R,{-})\:R\Modln\allowbreak\rarrow R\Modlc$
factorizes (uniquely) through the localization functor $R\Modln\rarrow
R\Modln/R\Modlz$, inducing a well-defined functor from the quotient category $R\Modln/R\Modlz\rarrow R\Modlc$.
 Then it follows from Lemma~\ref{c-unitalization-lemma} that
the resulting functor is left adjoint to the composition
$R\Modlc\rarrow R\Modln\rarrow R\Modln/R\Modlz$.
 It remains to show that the adjunction morphisms in $R\Modlc$
and $R\Modln/R\Modlz$   are isomorphisms.
 The adjunction counit $\Hom_R(R,P)\rarrow P$ (which is only defined
for $P\in R\Modlc$) is an isomorphism for every $P\in R\Modlc$
by construction.
 The adjunction unit $Q\rarrow\Hom_R(R,Q)$ is an isomorphism in
the quotient category $R\Modln/R\Modlz$ for every left $R$\+module $Q$
by Lemma~\ref{unitality-comparisons-have-null-co-kernels}(b).
 Finally, the functor $\Hom_R(R,{-})\:R\Modln\rarrow R\Modlc$ is exact
as the composition of two exact functors $R\Modln\rarrow R\Modln/
R\Modlz\rarrow R\Modlc$.
\end{proof}

\begin{cor} \label{main-three-abelian-category-equivalence-cor}
 Let $R$ be a t\+unital ring.
 Then there is a natural equivalence of three abelian categories
\begin{equation} \label{main-three-abelian-category-equivalence-eqn}
 R\Modlt\simeq R\Modln/R\Modlz\simeq R\Modlc.
\end{equation}
 In particular, the functors $R\ot_R{-}$ and\/ $\Hom_R(R,{-})$ restrict
to mutually inverse equivalences
\begin{equation} \label{t-unital-c-unital-equivalence-eqn}
 \xymatrix{
  {\Hom_R(R,{-})\: R\Modlt} \ar@{=}[rr]
  && {R\Modlc\,:\!R\ot_R{-}\,.}
 }
\end{equation}
\end{cor}

\begin{proof}
 The first assertion is obtained by combining the two parts of
Theorem~\ref{main-two-category-equivalences-thm}, and the second
assertion follows by looking into the constructions of two category
equivalences in Theorem~\ref{main-two-category-equivalences-thm}.
 This is sufficient to prove the corollary.

 It is instructive, however, to consider the adjunction constructions
for the two mutually inverse
equivalences~\eqref{t-unital-c-unital-equivalence-eqn}.
 Here one observes an explicit \emph{ambidextrous
adjunction}~\cite{CGN,Lau}, meaning a pair of functors that are adjoint
to each other on both sides.
 This ambidextrous adjunction is easier to establish than
the full assertions of the corollary.

 On the one hand, there is a natural adjunction where the functor
$\Hom_R(R,{-})$ is the right adjoint, and the functor $R\ot_R{-}$
is the left adjoint:
\begin{equation} \label{t-unital-c-unital-first-adjunction-eqn}
 \xymatrix{
  {\Hom_R(R,{-})\: R\Modlt} \ar@<-2pt>[rr]
  && {R\Modlc\,:\!R\ot_R{-}\,.} \ar@<-4pt>[ll]
 }
\end{equation}
 Indeed, consider the pair of adjoint endofunctors on the category of
unital left $\widetilde R$\+modules induced by
the $\widetilde R$\+$\widetilde R$\+bimodule $R$,
\begin{equation} \label{R-adjunction-eqn}
 \xymatrix{
  {\Hom_{\widetilde R}(R,{-})\: \widetilde R\Modl} \ar@<-2pt>[rr]
  && {\widetilde R\Modl\,:\!R\ot_{\widetilde R}{-}\,.} \ar@<-4pt>[ll]
 }
\end{equation}
 This is a particular case of the adjunction expressed by
the natural isomorphism~\eqref{Hom-associativity} from the introduction
or~\eqref{nonunital-two-rings-Hom-associativity} from
Section~\ref{c-unital-secn}.
 The adjunction~\eqref{t-unital-c-unital-first-adjunction-eqn}
is obtained by restricting the adjunction~\eqref{R-adjunction-eqn} to
the full subcategories $R\Modlt\subset R\Modln$ and
$R\Modlc\subset R\Modln\simeq\widetilde R\Modl$.

 On the other hand, there is a natural adjunction where the functor
$\Hom_R(R,{-})$ is the left adjoint, and the functor $R\ot_R{-}$
is the right adjoint:
\begin{equation} \label{t-unital-c-unital-second-adjunction-eqn}
 \xymatrix{
  {\Hom_R(R,{-})\: R\Modlt} \ar@<4pt>[rr]
  && {R\Modlc\,:\!R\ot_R{-}\,.} \ar@<2pt>[ll]
 }
\end{equation}
 The adjunction~\eqref{t-unital-c-unital-second-adjunction-eqn}
is the composition of the two adjunctions
$$
 \xymatrix{
  {R\Modlt} \ar@<4pt>[rr]^{\mathrm{inclusion}}
  && {R\Modln} \ar@<2pt>[ll]^{R\ot_R{-}}
  \ar@<4pt>[rr]^{\Hom_R(R,{-})}
  && {R\Modlc} \ar@<2pt>[ll]^{\mathrm{inclusion}}
 }
$$
provided by Lemmas~\ref{t-unitalization-lemma}
and~\ref{c-unitalization-lemma}.
\end{proof}

\begin{rem}
 It may be also instructive to consider the unit and counit morphisms
of the two adjunctions~\eqref{t-unital-c-unital-first-adjunction-eqn}
and~\eqref{t-unital-c-unital-second-adjunction-eqn}.
 For any left $R$\+module $M$, the evaluation map
$R\ot_R\Hom_R(R,M)\rarrow M$ is the counit of
the adjunction~\eqref{R-adjunction-eqn}.
 In the particular case when $M$ is a t\+unital $R$\+module,
the same evaluation map is the counit of
the adjunction~\eqref{t-unital-c-unital-first-adjunction-eqn}.
 On the other hand, applying the functor $R\ot_R{-}$ to
the natural map $M\rarrow\Hom_R(R,M)$, we obtain a natural map
$R\ot_RM\rarrow R\ot_R\Hom_R(R,M)$, which in the case of a t\+unital
$R$\+module $M$ provides a natural map $M\rarrow R\ot_R\Hom_R(R,M)$.
 The latter map is the unit of
the adjunction~\eqref{t-unital-c-unital-second-adjunction-eqn}.
 One can easily check that the composition $M\rarrow R\ot_R\Hom_R(R,M)
\rarrow M$ is the identity map.

 Similarly, for any left $R$\+module $P$, the natural map
$P\rarrow\Hom_R(R,\>R\ot_RP)$ is the unit of
the adjunction~\eqref{R-adjunction-eqn}.
 In the particular case when $P$ is a c\+unital $R$\+module,
the same map is the unit of
the adjunction~\eqref{t-unital-c-unital-first-adjunction-eqn}.
 On the other hand, applying the functor $\Hom_R(R,{-})$ to
the multiplication map $R\ot_RP\rarrow P$, we obtain a natural map
$\Hom_R(R,\>R\ot_RP)\rarrow\Hom_R(R,P)$, which in the case of
a c\+unital $R$\+module $P$ provides a natural map
$\Hom_R(R,\>R\ot_RP)\rarrow P$.
 The latter map is the counit of
the adjunction~\eqref{t-unital-c-unital-second-adjunction-eqn}.
 One can easily check that the composition $P\rarrow\Hom_R(R,\>R\ot_RP)
\rarrow P$ is the identity map.
\end{rem}

\Section{Rings Arising from Small Preadditive Categories}
\label{rings-from-categories-secn}

 Let $R$ be a (nonunital) ring.
 A \emph{family of orthogonal idempotents} $(e_x\in R)_{x\in X}$ in $R$
is a family of elements such that $e_x^2=e_x$ for all $x\in X$ and
$e_xe_y=0$ for all $x\ne y\in X$.
 A family of orthogonal idempotents $(e_x)_{x\in X}$ in $R$ is said
to be \emph{complete} if
$R=\bigoplus_{x\in X}R e_x=\bigoplus_{x\in X}e_x R$,
or equivalently, $R=\bigoplus_{x,y\in X} e_yRe_x$.
 Following the terminology of~\cite{Nys} (going back to~\cite{Ful}),
we will say that $R$ is \emph{a ring with enough idempotents} if
there is a chosen complete family of orthogonal idempotents in~$R$.
 
 The following lemma is easy and well-known~\cite{Nys}.

\begin{lem}
 Any ring with enough idempotents in s\+unital
(and consequently, t\+unital).
\end{lem}

\begin{proof}
 The point is that for any element $r\in R$ one can choose a finite
subset $Z\subset X$ such that the (idempotent) element
$e=\sum_{z\in Z}e_z\in R$ satisfies $er=r=re$.
 The assertion in parentheses follows by
Corollary~\ref{s-unital-implies-t-unital-for-rings}.
\end{proof}

\begin{lem} \label{with-enough-idempotents-is-projective}
 Let $R$ be a ring with enough idempotents and $\widetilde R=
\boZ\oplus R$ be its unitalization.
 Then $R$ is an (infinitely generated) projective unital left
$\widetilde R$\+module and projective unital right
$\widetilde R$\+module.
 Specifically, one has
\begin{equation} \label{R-decomposed-as-left-right-module}
 R=\bigoplus\nolimits_{x\in X}\widetilde R e_x
 =\bigoplus\nolimits_{x\in X}e_x\widetilde R,
\end{equation}
where, for every $x\in X$, the left $\widetilde R$\+module
$\widetilde R e_x=Re_x$ is projective with a single generator~$e_x$,
and the right $\widetilde R$\+module $e_x\widetilde R=e_xR$ is
projective with a single generator~$e_x$.
\end{lem}

\begin{proof}
 Notice that one has $\widetilde Rr=Rr$ and $r\widetilde R=rR$
for any s\+unital ring $R$ and element $r\in R$.
 With this simple observation, all the assertions of the lemma
follow immediately from the definition of a ring with enough
idempotents.
\end{proof}

\begin{cor} \label{t-unital-s-unital-modules-abelian-categories}
\textup{(a)} Let $R$ be a t\+unital ring.
 Then the categories of t\+unital and c\+unital left $R$\+modules
$R\Modlt$ and $R\Modlc$ are abelian. \par
\textup{(b)} Let $R$ be a left s\+unital ring (for example, $R$ is
a ring with enough idempotents).
 Then the full subcategory of t\+unital left $R$\+modules $R\Modlt$ is
closed under submodules, quotients, extensions, and direct sums in
$R\Modln$.
 Hence, in particular, the fully faithful inclusion functor of abelian
categories $R\Modlt\rarrow R\Modln$ is exact. \par
\textup{(c)} Let $R$ be a ring with enough idempotents.
 Then the full subcategory of c\+unital $R$\+modules $R\Modlc$
is closed under kernels, cokernels, extensions, and direct products
in $R\Modln$.
 Hence, in particular, the fully faithful inclusion functor of abelian
categories $R\Modlc\rarrow R\Modln$ is exact. 
\end{cor}

\begin{proof}
 Part~(a) is a part of
Theorem~\ref{main-two-category-equivalences-thm}(a\+-b).
 Part~(b) is Proposition~\ref{s-unital-hereditary-torsion}
together with Corollary~\ref{s-unital-iff-t-unital-for-modules}(a).
 Part~(c) follows from Lemma~\ref{c-unital-modules-closure-properties}
together with Remark~\ref{Ext-perpendicularity-remark} and the first
assertion of Lemma~\ref{with-enough-idempotents-is-projective}.
\end{proof}

 Let $\sE$ be a small preadditive category, i.~e., a small category
enriched in abelian groups.
 This means that, for every pair of objects $x$, $y\in\sE$, there is
the abelian group of morphisms $\Hom_\sE(x,y)$, and the composition maps
$$
 \Hom_\sE(y,z)\ot_\boZ\Hom_\sE(x,y)\lrarrow\Hom_\sE(x,z)
$$
are defined for all objects $x$, $y$, $z\in\sE$ in such a way that
the usual associativity and unitality axioms (for identity elements
$\id_x\in\Hom_\sE(x,x)$) are satisfied.

 Then the abelian group $R_\sE=\bigoplus_{x,y\in\sE}\Hom_\sE(x,y)$
has a natural structure of (nonunital) ring with the multiplication
map given by the composition of morphisms in~$\sE$.
 The following lemma is certainly standard and well-known.

\begin{lem} \label{rings-with-enough-idempotents-are-categories}
 A ring $R$ has enough idempotents if and only if it has the form
$R=R_\sE$ for some small preadditive category\/~$\sE$.
\end{lem}

\begin{proof}
 Let $R$ be a ring with enough idempotents, and let
$(e_x\in R)_{x\in X}$ be its complete family of orthogonal idempotents.
 Then $\sE$ is the category with the set of objects $X$ and the groups
of morphisms $\Hom_\sE(x,y)=e_yRe_x$.
 Conversely, if $R=R_\sE$ for a small preadditive category $\sE$, then
the identity elements $\id_x\in\Hom_\sE(x,x)$, viewed as elements of
$R_\sE$, form a complete family of orthogonal idempotents in~$R$.
\end{proof}

 Let $\sE$ be a small preadditive category.
 By a \emph{left\/ $\sE$\+module} we mean a covariant additive functor
$\sE\rarrow\Ab$ from $\sE$ to the category of abelian groups.
 In other words, a left $\sE$\+module $\sM$ is a rule assigning to every
object $x\in\sE$ an abelian group $\sM(x)$, and to every pair of
objects $x$, $y\in\sE$ the \emph{action map}
$$
 \Hom_\sE(x,y)\ot_\boZ\sM(x)\lrarrow\sM(y)
$$
in such a way that the usual associativity and unitality axioms are
satisfied.

 Similarly, a \emph{right\/ $\sE$\+module} $\sN$ is a contravariant
additive functor $\sE^\sop\rarrow\Ab$.
 In other words, it is a rule assigning to every object $x\in\sE$
an abelian group $\sN(x)$, and to every pair of objects $x$, $y\in\sE$
the action map
$$
 \Hom_\sE(x,y)\ot_\boZ\sN(y)\lrarrow\sN(x)
$$
satisfying the usual associativity and unitality axioms.
 We denote the abelian categories of left and right $\sE$\+modules by
$\sE\Modl$ and $\Modr\sE$, respectively.

 The next proposition is also well-known.
 One relevant reference is~\cite[Section~7]{Mit}.

\begin{prop} \label{t-unital-over-ring-with-enough-idempotents}
 Let\/ $\sE$ be a small preadditive category and $R=R_\sE$ be
the related ring with enough idempotents.
 In this context: \par
\textup{(a)} The category of t\+unital left $R$\+modules is naturally
equivalent to the category of left\/ $\sE$\+modules,
$$
 R_\sE\Modlt\simeq\sE\Modl.
$$
 The equivalence assigns to a left\/ $\sE$\+module\/ $\sM$
the t\+unital left $R$\+module $M=\bigoplus_{x\in\sE}\sM(x)$, with
the action of $R$ in $M$ induced by the action of\/ $\sE$ in\/~$\sM$.
\par
\textup{(b)} The category of t\+unital right $R$\+modules is naturally
equivalent to the category of right\/ $\sE$\+modules,
$$
 \Modrt R_\sE\simeq\Modr\sE.
$$
 The equivalence assigns to a right\/ $\sE$\+module\/ $\sN$
the t\+unital right $R$\+module $N=\bigoplus_{x\in\sE}\sN(x)$, with
the action of $R$ in $N$ induced by the action of\/ $\sE$ in\/~$\sN$.
\end{prop}

\begin{proof}
 Let us explain part~(a).
 By Lemmas~\ref{with-enough-idempotents-is-projective}
and~\ref{rings-with-enough-idempotents-are-categories}, we have
an isomorphism of right $R$\+modules $R\simeq\bigoplus_{x\in X} e_xR$,
where $X$ is the set of objects of $\sE$ and $e_x=\id_x$.
 Consequently, for any left $R$\+module $M$ we have a natural
direct sum decomposition
$$
 R\ot_R M\simeq\bigoplus\nolimits_{x\in X} e_xR\ot_R M=
 \bigoplus\nolimits_{x\in X} e_xM.
$$
 The second equality holds because $e_xR=e_x\widetilde R$.
 When the $R$\+module $M$ is t\+unital, we obtain a direct sum
decomposition $M=\bigoplus_{x\in X}e_xM$, allowing to recover
the corresponding $\sE$\+module $\sM$ by the rule $\sM(x)=e_xM$
for all $x\in\sE$.
\end{proof}

\begin{prop} \label{c-unital-over-ring-with-enough-idempotents}
 Let\/ $\sE$ be a small preadditive category and $R=R_\sE$ be
the related ring with enough idempotents.
 Then the category of c\+unital left $R$\+modules is naturally
equivalent to the category of left\/ $\sE$\+modules,
$$
 R_\sE\Modlc\simeq\sE\Modl.
$$
 The equivalence assigns to a left\/ $\sE$\+module\/ $\sP$
the c\+unital left $R$\+module $P=\prod_{x\in\sE}\sP(x)$, with
the action of $R$ in $P$ induced by the action of\/ $\sE$ in\/~$\sP$.
\end{prop}

\begin{proof}
 By Lemmas~\ref{with-enough-idempotents-is-projective}
and~\ref{rings-with-enough-idempotents-are-categories}, we have
an isomorphism of left $R$\+modules $R\simeq\bigoplus_{x\in X} Re_x$,
where $X$ is the set of objects of $\sE$ and $e_x=\id_x$.
 Consequently, for any left $R$\+module $P$ we have a natural
direct product decomposition
$$
 \Hom_R(R,P)\simeq\prod\nolimits_{x\in X}\Hom_R(Re_x,P)=
 \prod\nolimits_{x\in X} e_xP.
$$
 The second equality holds because $Re_x=\widetilde Re_x$.
 When the $R$\+module $P$ is c\+unital, we obtain a direct product
decomposition $P=\prod_{x\in X}e_xP$, allowing to recover
the corresponding $\sE$\+module $\sP$ by the rule $\sP(x)=e_xP$
for all $x\in\sE$.
\end{proof}

\begin{prop} \label{four-category-equivalence}
 Let\/ $\sE$ be a small preadditive category and $R=R_\sE$ be
the related ring with enough idempotents.
 Then there is a commutative square diagram of category equivalences
provided by
Propositions~\ref{t-unital-over-ring-with-enough-idempotents}(a)
and~\ref{c-unital-over-ring-with-enough-idempotents} together
with Corollary~\ref{main-three-abelian-category-equivalence-cor},
\begin{equation}
\begin{gathered}
 \xymatrix{
  && {R\Modln/R\Modlz} \ar@{=}[lld] \ar@{=}[rrd] \\
  R_\sE\Modlt \ar@{=}[rrd] &&&& R_\sE\Modlc \ar@{=}[lld] \\
  && \sE\Modl
 }
\end{gathered}
\end{equation}
\end{prop}

\begin{proof}
 It suffices to check that the composition
$$
 R_\sE\Modlt\lrarrow\sE\Modl\lrarrow R_\sE\Modlc
$$
of the category equivalences from
Propositions~\ref{t-unital-over-ring-with-enough-idempotents}(a)
and~\ref{c-unital-over-ring-with-enough-idempotents} is isomorphic to
the Hom functor from~\eqref{t-unital-c-unital-equivalence-eqn},
and/or that the composition of category equivalences
$$
 R_\sE\Modlc\lrarrow\sE\Modl\lrarrow R_\sE\Modlt
$$
is isomorphic to the tensor product functor
from~\eqref{t-unital-c-unital-equivalence-eqn}.
 For this purpose, one observes that, for any left $R$\+module $L$,
one has
$$
 e_xR\ot_RL \simeq e_xL \simeq \Hom_R(Re_x,L).
$$
\end{proof}

\Section{More on t-Unital and c-Unital Modules}
\label{more-on-t-c-unital-secn}

 Let $R$ be a t\+unital ring.
 According to
Examples~\ref{one-variable-non-c-unital-cokernel-counterex},
\ref{t-unital-ring-non-t-unital-kernel-counterex},
and~\ref{t-unital-ring-non-c-unital-cokernel-counterex}, the full
subcategory of t\+unital modules $R\Modlt$ \emph{need not} be closed
under kernels in the abelian category $R\Modln$, and the full
subcategory of c\+unital modules $R\Modlc$ \emph{need not} be closed
under cokernels in $R\Modln$.
 Nevertheless,
Corollary~\ref{t-unital-s-unital-modules-abelian-categories}(a) tells
that the categories $R\Modlt$ and $R\Modlc$ are abelian; so all kernels
and cokernels exist in them.
 The following two lemmas and proposition explain more.

\begin{lem} \label{t-unital-modules-kernels}
 Let $R$ be a t\+unital ring.
 Then all limits and colimits exist in the category of t\+unital
modules $R\Modlt$.
\end{lem}

\begin{proof}
 The abelian category $R\Modln\simeq\widetilde R\Modl$ obviously has
all limits and colimits.
 By Lemma~\ref{t-unital-modules-closure-properties}, the full
subcategory $R\Modlt$ is closed under colimits in $R\Modln$.
 So all colimits exist in $R\Modlt$ and can be computed as
the colimits in $R\Modln$.
 On the other hand, Lemma~\ref{t-unitalization-lemma} tells that
the functor $R\ot_R{-}\,\:R\Modln\rarrow R\Modlt$ is right adjoint
to the inclusion $R\Modlt\rarrow R\Modln$.
 Consequently, the functor $R\ot_R{-}\,\:R\Modln\rarrow R\Modlt$
preserves limits.
 Now, in order to compute the limit of a diagram $(M_\xi)_{\xi\in\Xi}$
in $R\Modlt$, it suffices to apply the functor $R\ot_R{-}$ to
the limit of the same diagram in $R\Modln$.
\end{proof}

\begin{lem} \label{c-unital-modules-cokernels}
 Let $R$ be a t\+unital ring.
 Then all limits and colimits exist in the category of c\+unital
modules $R\Modlc$.
\end{lem}

\begin{proof}
 By Lemma~\ref{c-unital-modules-closure-properties}, the full
subcategory $R\Modlc$ is closed under limits in $R\Modln$.
 So all limits exist in $R\Modlc$ and can be computed as
the limits in $R\Modln$.
 On the other hand, Lemma~\ref{c-unitalization-lemma} tells that
the functor $\Hom_R(R,{-})\:R\Modln\rarrow R\Modlc$ is left adjoint
to the inclusion $R\Modlc\rarrow R\Modln$.
 Consequently, the functor $\Hom_R(R,{-})\:R\Modln\rarrow R\Modlc$
preserves colimits.
 Now, in order to compute the colimit of a diagram $(P_\xi)_{\xi\in\Xi}$
in $R\Modlc$, it suffices to apply the functor $\Hom_R(R,{-})$ to
the colimit of the same diagram in $R\Modln$.
\end{proof}

 A study of abelian categories of the form $R\Modln/R\Modlz$ (for
idempotent rings~$R$) can be found in the note of Roos~\cite{Roos}
(cf.~\cite[Section~6.3]{Quil}).
 The following result is a particular case
of~\cite[Th\'eor\`eme~1]{Roos}.

\begin{prop}[\cite{Roos}] \label{grothendieck-with-exact-products}
 Let $R$ be a t\+unital ring.
 Then the abelian category $R\Modlt\simeq R\Modln/R\Modlz\simeq
R\Modlc$ is a Grothendieck abelian category with exact functors of
infinite product.
\end{prop}

\begin{proof}
 The equivalences of abelian categories
$R\Modlt\simeq R\Modln/R\Modlz\simeq R\Modlc$ hold by
Corollary~\ref{main-three-abelian-category-equivalence-cor}.
 Furthermore, the Serre subcategory $R\Modlz$ is closed under direct
sums and direct products in $R\Modln$, as mentioned in the beginning
of Section~\ref{equivalence-secn}; and the localization functor
$R\Modln\rarrow R\Modln/R\Modlz$ has adjoints on both sides, as
per Remark~\ref{Giraud-remark}.

 It is well-known that for any Grothendieck abelian category $\sA$
with a Serre subcategory $\sS\subset\sA$ such that the localization
functor $\sA\rarrow\sA/\sS$ admits a right adjoint, the quotient
category $\sA/\sS$ is also Grothendieck~\cite[Proposition~8
in Section~III.4]{Gab}, \cite[Proposition~X.1.3]{Sten}.
 In particular, the direct limit functors in $\sA/\sS$ are exact,
since they are exact in $\sA$ and the localization functor
$\sA\rarrow\sA/\sS$, being a left adjoint, preserves direct limits.
 Similarly, if the functor $\sA\rarrow\sA/\sS$ has a left adjoint,
then it preserves direct products; and if the direct products are
exact in $\sA$, then it follows that they are also exact in~$\sA/\sS$.

 In fact, it is not difficult to prove directly that for any abelian
category $\sA$ with exact functors of coproduct and a Serre subcategory
$\sS\subset\sA$ closed under coproducts, the localization functor
$\sA\rarrow\sA/\sS$ preserves coproducts (and takes any set of
generators of $\sA$ to a set of generators of $\sA/\sS$).
 Dually, if an abelian category $\sA$ has exact products and a Serre
subcategory $\sS\subset\sA$ is closed under products, then
the localization functor $\sA\rarrow\sA/\sS$ preserves products.
 See, e.~g., \cite[Lemma~1.5]{BN} or~\cite[Lemma~3.2.10]{Neem} for
a similar argument in the triangulated category realm.
 (Cf.\ the discussion in~\cite[Section~6.4 and Proposition~6.5]{Quil}
for module categories.)
\end{proof}

\begin{cor} \label{closed-under-kernels-iff-flat}
 Let $R$ be a t\+unital ring and $\widetilde R=\boZ\oplus R$ be its
unitalization.
 Then the full subcategory $R\Modlt$ is closed under kernels in
$R\Modln$ \emph{if and only if} the unital right $\widetilde R$\+module
$R$ is flat.
\end{cor}

\begin{proof}
 The ``if'' assertion was explained in
Remark~\ref{Tor-perpendicularity-remark}; it does not depend on
the assumption that $R$ is t\+unital.
 The ``only if'' claim is a corollary of
Theorem~\ref{main-two-category-equivalences-thm}(a).
 Assume that the full subcategory $R\Modlt$ is closed under kernels
in $R\Modln$; then the inclusion functor $R\Modlt\rarrow R\Modln$
preserves kernels.
 The localization functor $R\Modln\rarrow R\Modln/R\Modlz\simeq
R\Modlt$ is exact, so it also preserves kernels.
 Thus the composition $R\Modln\rarrow R\Modlt\rarrow R\Modln$ preserves
kernels.
 But this composition is the tensor product functor $R\ot_R{-}\,\:
R\Modln\rarrow R\Modln$, that is, in other notation, the tensor
product functor $R\ot_{\widetilde R}{-}\,\:
\widetilde R\Modl\rarrow\widetilde R\Modl$.
 Alternatively, one could refer to Remark~\ref{t-flat-long-remarks}(3)
and Corollary~\ref{t-flat-cor} below.
\end{proof}

\begin{cor} \label{left-s-unital-right-flat}
 Let $R$ be a left s\+unital ring and $\widetilde R=\boZ\oplus R$ be its
unitalization.
 Then the unital right $\widetilde R$\+module $R$ is flat.
\end{cor}

\begin{proof}
 By Corollary~\ref{s-unital-implies-t-unital-for-rings}, the ring $R$
is t\+unital.
 By Corollary~\ref{s-unital-iff-t-unital-for-modules}(a), a left
$R$\+module is t\+unital if and only if it is s\+unital,
$R\Modlt=R\Modls$.
 Proposition~\ref{s-unital-hereditary-torsion} tells that the full
subcategory $R\Modls$ of s\+unital left $R$\+modules is closed under
submodules (hence also under kernels) in $R\Modln$.
 Finally, it remains to apply the ``only if'' assertion of
Corollary~\ref{closed-under-kernels-iff-flat}.
\end{proof}

\begin{cor} \label{closed-under-cokernels-iff-projective}
  Let $R$ be a t\+unital ring and $\widetilde R=\boZ\oplus R$ be its
unitalization.
 Then the full subcategory $R\Modlc$ is closed under cokernels in
$R\Modln$ \emph{if and only if} the unital left $\widetilde R$\+module
$R$ is projective.
\end{cor}

\begin{proof}
 The ``if'' assertion was explained in
Remark~\ref{Ext-perpendicularity-remark}; it does not depend on
the assumption that $R$ is t\+unital.
 The ``only if'' claim is a corollary of
Theorem~\ref{main-two-category-equivalences-thm}(b).
 Assume that the full subcategory $R\Modlc$ is closed under cokernels
in $R\Modln$; then the inclusion functor $R\Modlc\rarrow R\Modln$
preserves cokernels.
 The localization functor $R\Modln\rarrow R\Modln/R\Modlz\simeq
R\Modlc$ is exact, so it also preserves cokernels.
 Thus the composition $R\Modln\rarrow R\Modlc\rarrow R\Modln$ preserves
cokernels.
 But this composition is the functor $\Hom_R(R,{-})\:
R\Modln\rarrow R\Modln$, that is, in other notation, the functor
$\Hom_{\widetilde R}(R,{-})\: \widetilde R\Modl\rarrow
\widetilde R\Modl$.
 Alternatively, one could refer to
Remark~\ref{c-projective-long-remarks}(3)
and Corollary~\ref{c-projective-cor} below.
\end{proof}

\begin{rem}
 The theory developed in Sections~\ref{t-unital-secn},
\ref{c-unital-secn}, \ref{equivalence-secn},
and~\ref{more-on-t-c-unital-secn} admits a far-reaching
generalization of the following kind.

 Let $\widetilde R$ be a unital ring and $R\subset\widetilde R$ be
a two-sided ideal.
 Let us say that a left $\widetilde R$\+module $M$ is
\emph{t\+$R$\+unital} if the natural map $R\ot_{\widetilde R}M
\rarrow M$ is an isomorphism, a right $\widetilde R$\+module $N$
is t\+$R$\+unital if the natural map $N\ot_{\widetilde R}R
\rarrow N$ is an isomorphism, and a left $\widetilde R$\+module $P$
is \emph{c\+$R$\+unital} if the natural map
$P\rarrow\Hom_{\widetilde R}(R,P)$ is an isomorphism.
 Let us say that the ideal $R\subset\widetilde R$ itself is
\emph{t\+$R$\+unital} if it is t\+$R$\+unital as a left/right
$\widetilde R$\+module, i.~e., the natural map
$R\ot_{\widetilde R}R\rarrow R$ is an isomorphism.

 All the results of Sections~\ref{t-unital-secn}, \ref{c-unital-secn},
\ref{equivalence-secn}, and~\ref{more-on-t-c-unital-secn}, as well as
of the next Section~\ref{flat-proj-inj-secn}, can be extended to
this more general context.
 We chose to restrict ourselves to the more narrow setting of nonunital
rings in these five sections in order to keep our exposition more
accessible and facilitate the comparison with
Sections~\ref{s-unital-secn} and~\ref{rings-from-categories-secn}.

 A detailed discussion in the context of idempotent two-sided ideals
in unital rings can be found in the manuscript~\cite{Quil}.
 Moreover, it is explained in~\cite[\S9]{Quil} that the resulting
theory does not depend on (in our notation) the chosen embedding of
a given nonunital ring $R$ into a unital ring $\widetilde R$
where $R$ is an ideal.
\end{rem}

\Section{t-Flat, c-Projective, and t-Injective Modules}
\label{flat-proj-inj-secn}

 We start with a discussion of c\+projectivity, then pass to
t\+injectivity, and in the end come to t\+flatness.

\begin{defn} \label{c-projective-defn}
 Let $R$ be a t\+unital ring.
 We will say that a left $R$\+module $Q$ is \emph{c\+projective} if
the covariant functor $\Hom_R(Q,{-})\:R\Modlc\rarrow\Ab$ preserves
cokernels (i.~e., takes cokernels in the category $R\Modlc$ to
cokernels in the category of abelian groups~$\Ab$).
\end{defn}

\begin{rems} \label{c-projective-long-remarks}
 (0)~Notice that, by the definition, a c\+projective $R$\+module
\emph{need not} be c\+unital.

 (1)~Furthermore, a projective nonunital $R$\+module (i.~e.,
a projective unital $\widetilde R$\+module) \emph{need not} be
c\+projective.
 For example, let $T$ be the commutative ring from
Example~\ref{rational-power-polynomials-two-variables-t-unital}
and $0\rarrow F\rarrow\widetilde F\rarrow S\rarrow0$ be
the short exact sequence of $T$\+modules from
Example~\ref{t-unital-ring-non-c-unital-cokernel-counterex}.
 So $F$ and $\widetilde F$ are c\+unital $T$\+modules, but $S$ is not.
 One can easily compute that $\Hom_T(T,S)=\Hom_S(S,S)=\widetilde S$,
where the basis elements $x^py^q\in T$ act by zero in $\widetilde S$
when $p>0$.
 According to the proof of Lemma~\ref{c-unital-modules-cokernels},
\,$\widetilde S$ is the cokernel of the morphism $F\rarrow\widetilde F$
in $T\Modlc$.
 Since the map $\widetilde F\rarrow\widetilde S$ is not surjective
(i.~e., not an epimorphism in $T\Modln$ or in~$\Ab$), it follows that
the free $\widetilde T$\+module $\widetilde T$ is \emph{not}
c\+projective as a $T$\+module.

 (2)~Quite generally, the free left $\widetilde R$\+module
$\widetilde R$ is c\+projective as an $R$\+module \emph{if and only if}
the full subcategory $R\Modlc$ is closed under cokernels in $R\Modln$.
 If this is the case (e.~g., this holds for all rings $R$ with enough
idempotents, see
Corollary~\ref{t-unital-s-unital-modules-abelian-categories}(c)),
then all projective nonunital $R$\+modules are c\+projective.

 (3)~The left $R$\+module $R$ is c\+projective if and only if
the left $R$\+module $\widetilde R$ is (cf.\
Proposition~\ref{c-projective-prop} below).
 Consequently, the left $R$\+module $R$ is c\+projective if and only if
the full subcategory $R\Modlc$ is closed under cokernels in $R\Modln$.
\end{rems}

\begin{lem} \label{c-projective-lemma}
 Let $R$ be a t\+unital ring and $Q$ be a t\+unital left $R$\+module.
 Then the $R$\+module $Q$ is c\+projective if and only if, for every
right exact sequence $L\rarrow M\rarrow E\rarrow0$ in $R\Modln$ with
$L$, $M\in R\Modlc$, the sequence of abelian groups $\Hom_R(Q,L)\rarrow
\Hom_R(Q,M)\rarrow\Hom_R(Q,E)\rarrow0$ is right exact.
\end{lem}

\begin{proof}
 Following the proof of Lemma~\ref{c-unital-modules-cokernels},
the c\+unital $R$\+module $\Hom_R(R,E)$ is the cokernel of
the morphism $L\rarrow M$ in $R\Modlc$.
 By the definition, c\+projectivity of $Q$ means right exactness of
the sequences of abelian groups $\Hom_R(Q,L)\rarrow\Hom_R(Q,M)
\rarrow\Hom_R(Q,\Hom_R(R,E))\rarrow0$.
 It remains to recall the natural isomorphisms of abelian groups
$\Hom_R(Q,\Hom_R(R,E))\simeq\Hom_R(R\ot_RQ,\>E)\simeq\Hom_R(Q,E)$
for a t\+unital $R$\+module~$Q$.
\end{proof}

\begin{prop} \label{c-projective-prop}
 Let $R$ be a t\+unital ring and $Q$ be a left $R$\+module.
 Then the following conditions are equivalent:
\begin{enumerate}
\item $Q$ is a c\+projective $R$\+module;
\item the c\+unital $R$\+module\/ $\Hom_R(R,Q)$ is c\+projective;
\item the t\+unital $R$\+module $R\ot_RQ$ is c\+projective.
\end{enumerate}
\end{prop}

\begin{proof}
 (1)\,$\Longleftrightarrow$\,(2) follows from
Lemma~\ref{c-unitalization-lemma}.

 (1)\,$\Longleftrightarrow$\,(3) holds due to the natural isomorphisms
$\Hom_R(R\ot_RQ,\>P)\simeq\Hom_R(Q,\Hom_R(R,P))\simeq\Hom_R(Q,P)$ for
all c\+unital left $R$\+modules~$P$.  {\hbadness=1200\par}
 
 Alternatively, both the equivalences follow from
Lemma~\ref{unitality-comparisons-have-null-co-kernels} and
Corollary~\ref{null-isomorphisms-inverted-by-tensor-Hom-with-t-c}(c).
\end{proof}

\begin{rem} \label{c-projective-short-remark}
 Clearly, a c\+unital left $R$\+module is c\+projective if and only if
it is a projective object in the abelian category $R\Modlc$.
 Now, in view of
Corollary~\ref{main-three-abelian-category-equivalence-cor},
Proposition~\ref{c-projective-prop} implies that a t\+unital
left $R$\+module is c\+projective if and only if it is a projective
object in the abelian category $R\Modlt$.
 Furthermore, by Theorem~\ref{main-two-category-equivalences-thm}(b),
a nonunital left $R$\+module is c\+projective if and only if it
represents a projective object in the quotient category
$R\Modln/R\Modlz$.
\end{rem}

 The following corollary is our version of the projective
modules/objects claim in~\cite[Proposition~6.2]{Quil}.

\begin{cor} \label{c-projective-cor}
 Let $R$ be a t\+unital ring and $Q$ be a t\+unital left $R$\+module.
 Then $Q$ is c\+projective if and only if it is projective as
an object of the category of nonunital modules $R\Modln$.
\end{cor}

\begin{proof}
 By Remark~\ref{c-projective-short-remark}, \,$Q$  is c\+projective
if and only if it represents a projective object in $R\Modln/R\Modlz$.
 As mentioned in Remark~\ref{Giraud-remark}, the inclusion functor
$R\Modln/R\Modlz\simeq R\Modlt\rarrow R\Modln$ is left adjoint to
the localization functor $R\Modln\rarrow R\Modln/R\Modlz$.
 Since the latter functor is exact, the former one takes projectives
to projectives.
 This proves the ``only if'' implication of the corollary.
 The ``if'' implication follows from the ``if'' assertion of
Lemma~\ref{c-projective-lemma}.
\end{proof}

 The next definition is dual-analogous to
Definition~\ref{c-projective-defn}.

\begin{defn}
 Let $R$ be a t\+unital ring.
 We will say that a left $R$\+module $J$ is \emph{t\+injective} if
the contravariant functor $\Hom_R({-},J)\:R\Modlt^\sop\rarrow\Ab$
takes the kernels in $R\Modlt$ to the cokernels in~$\Ab$.
\end{defn}

\begin{rems} \label{t-injective-long-remarks}
 (0)~Notice that, by the definition, a t\+injective $R$\+module
\emph{need not} be t\+unital.

 (1)~Furthermore, an injective nonunital $R$\+module (i.~e.,
an injective unital $\widetilde R$\+module) \emph{need not} be
t\+injective.
 Quite generally, the cofree left $\widetilde R$\+module
$\Hom_\boZ(\widetilde R,\boQ/\boZ)$ is t\+injective as an $R$\+module
\emph{if and only if} the full subcategory $R\Modlt$ is closed under
kernels in $R\Modln$.
 Example~\ref{t-unital-ring-non-t-unital-kernel-counterex} shows that
this is not always the case.

 (2) If the full subcategory $R\Modlt$ is closed under kernels in
$R\Modln$ (e.~g., this holds for left s\+unital rings~$R$;
see Proposition~\ref{s-unital-hereditary-torsion}
and Corollary~\ref{s-unital-iff-t-unital-for-modules}(a)),
then all injective nonunital $R$\+modules are t\+injective.
\end{rems}

\begin{lem} \label{t-injective-lemma}
 Let $R$ be a t\+unital ring and $J$ be a c\+unital left $R$\+module.
 Then the $R$\+module $J$ is t\+injective if and only if, for every
left exact sequence $0\rarrow K\rarrow L\rarrow M$ in $R\Modln$ with
$L$, $M\in R\Modlt$, the sequence of abelian groups $\Hom_R(M,J)
\rarrow\Hom_R(L,J)\rarrow\Hom_R(K,J)\rarrow0$ is right exact.
\end{lem}

\begin{proof}
 Following the proof of Lemma~\ref{t-unital-modules-kernels},
the t\+unital $R$\+module $R\ot_RK$ is the kernel of the morphism
$L\rarrow M$ in $R\Modlt$.
 By the definition, t\+injectivity of $J$ means right exactness of
the sequences of abelian groups  $\Hom_R(M,J)\rarrow\Hom_R(L,J)
\rarrow\Hom_R(R\ot_RK,\>J)\rarrow0$.
 It remains to recall the natural isomorphisms of abelian groups
$\Hom_R(R\ot_RK,\>J)\simeq\Hom_R(K,\Hom_R(R,J))\simeq
\Hom_R(K,J)$ for a c\+unital $R$\+module~$J$.
\end{proof}

\begin{prop} \label{t-injective-prop}
 Let $R$ be a t\+unital ring and $J$ be a left $R$\+module.
 Then the following conditions are equivalent:
\begin{enumerate}
\item $J$ is a t\+injective $R$\+module;
\item the t\+unital $R$\+module $R\ot_RJ$ is t\+injective;
\item the c\+unital $R$\+module\/ $\Hom_R(R,J)$ is t\+injective.
\end{enumerate}
\end{prop}

\begin{proof}
 (1)\,$\Longleftrightarrow$\,(2) follows from
Lemma~\ref{t-unitalization-lemma}.

 (1)\,$\Longleftrightarrow$\,(3) holds due to the natural isomorphisms
$\Hom_R(M,\Hom_R(R,J))\simeq\Hom_R(R\ot_RM,\>J)\simeq\Hom_R(M,J)$ for
all t\+unital left $R$\+modules~$M$.

 Alternatively, both the equivalences follow from
Lemma~\ref{unitality-comparisons-have-null-co-kernels} and
Corollary~\ref{null-isomorphisms-inverted-by-tensor-Hom-with-t-c}(b).
\end{proof}

\begin{rem} \label{t-injective-short-remark}
 Clearly, a t\+unital left $R$\+module is t\+injective if and only if
it is an injective object in the abelian category $R\Modlt$.
 Now, in view of
Corollary~\ref{main-three-abelian-category-equivalence-cor},
Proposition~\ref{t-injective-prop} implies that a c\+unital
left $R$\+module is t\+injective if and only if it is an injective
object in the abelian category $R\Modlc$.
 Furthermore, by Theorem~\ref{main-two-category-equivalences-thm}(a),
a nonunital left $R$\+module is t\+injective if and only
if it represents an injective object in the quotient category
$R\Modln/R\Modlz$.
\end{rem}

 The next corollary is our version of the injective
modules/objects claim in~\cite[Proposition~6.2]{Quil}.

\begin{cor} \label{t-injective-cor}
 Let $R$ be a t\+unital ring and $J$ be a c\+unital left $R$\+module.
 Then $J$ is t\+injective if and only if it is injective as
an object of the category of nonunital modules $R\Modln$.
\end{cor}

\begin{proof}
 This is a particular case of~\cite[Proposition~X.1.4]{Sten}.
 By Remark~\ref{t-injective-short-remark}, \,$J$  is t\+injective
if and only if it represents an injective object in $R\Modln/R\Modlz$.
 As mentioned in Remark~\ref{Giraud-remark}, the inclusion functor
$R\Modln/R\Modlz\simeq R\Modlc\rarrow R\Modln$ is right adjoint to
the localization functor $R\Modln\rarrow R\Modln/R\Modlz$.
 Since the latter functor is exact, the former one takes injectives
to injectives.
 This proves the ``only if'' implication of the corollary.
 The ``if'' implication follows from the ``if'' assertion of
Lemma~\ref{t-injective-lemma}.
\end{proof}

 Finally, we come to the last main definition of this section.

\begin{defn}
 Let $R$ be a t\+unital ring.
 We will say that a left $R$\+module $F$ is \emph{t\+flat} if
the covariant functor ${-}\ot_RF\:\Modrt R\rarrow\Ab$ preserves
kernels (i.~e., takes kernels in $\Modrt R$ to kernels in~$\Ab$).
 The definition of a t\+flat right $R$\+module is similar.
\end{defn}

\begin{rems} \label{t-flat-long-remarks}
 (0)~Notice that, by the definition, a t\+flat $R$\+module \emph{need
not} be t\+unital.

 (1)~Furthermore, a flat nonunital $R$\+module (i.~e., a flat
unital $\widetilde R$\+module) \emph{need not} be t\+flat.
 Quite generally, the free left $\widetilde R$\+module $\widetilde R$
is t\+flat as a left $R$\+module \emph{if and only if} the full
subcategory of t\+flat right $R$\+modules $\Modrt R$ is closed under
kernels in $\Modrn R$.
 Example~\ref{t-unital-ring-non-t-unital-kernel-counterex} shows that
this is not always the case.

 (2)~If the full subcategory $\Modrt R$ is closed under kernels in
$\Modrn R$, then all flat nonunital $R$\+modules are t\+flat.

 (3)~The left $R$\+module $R$ is t\+flat if and only if the left
$R$\+module $\widetilde R$ is (cf.\ Proposition~\ref{t-flat-prop}
below).
 Consequently, the left $R$\+module $R$ is t\+flat if and only if
the full subcategory $\Modrt R$ is closed under kernels in $\Modrn R$.
\end{rems}

\begin{lem} \label{t-flat-lemma}
 Let $R$ be a t\+unital ring and $F$ be a t\+unital left $R$\+module.
 Then the $R$\+module $F$ is t\+flat if and only if, for every left
exact sequence\/ $0\rarrow K\rarrow L\rarrow M$ in $\Modrn R$ with
$L$, $M\in\Modrt R$, the sequence of abelian groups\/
$0\rarrow K\ot_RF\rarrow L\ot_RF\rarrow M\ot_RF$ is left exact.
\end{lem}

\begin{proof}
 Following the proof of Lemma~\ref{t-unital-modules-kernels},
the t\+unital right $R$\+module $K\ot_RR$ is the kernel of
the morphism $L\rarrow M$ in $\Modrt R$.
 By the definition, t\+flatness of $F$ means left exactness of
the sequences of abelian groups $0\rarrow (K\ot_RR)\ot_RF\rarrow
L\ot_RF\rarrow M\ot_RF$.
 It remains to recall the natural isomorphisms of abelian groups
$(K\ot_RR)\ot_RF\simeq K\ot_R(R\ot_RF)\simeq K\ot_RF$ for
a t\+unital $R$\+module~$F$.
\end{proof}

\begin{prop} \label{t-flat-prop}
 Let $R$ be a t\+unital ring and $F$ be a left $R$\+module.
 Then the following conditions are equivalent:
\begin{enumerate}
\item $F$ is a t\+flat $R$\+module;
\item the t\+unital $R$\+module $R\ot_RF$ is t\+flat;
\item the c\+unital $R$\+module\/ $\Hom_R(R,F)$ is t\+flat.
\end{enumerate}
\end{prop}

\begin{proof}
 More generally, the claim is that t\+flatness of a left $R$\+module
$F$ only depends on the image of $F$ in the Serre quotient category
$R\Modln/R\Modlz$.
 In other words, if $F\rarrow G$ is a morphism of left $R$\+modules
with null kernel and cokernel, that the $R$\+module $F$ is t\+flat
if and only if the $R$\+module $G$ is.
 This claim follows from
Corollary~\ref{null-isomorphisms-inverted-by-tensor-Hom-with-t-c}(a),
and it implies the desired equivalences by
Lemma~\ref{unitality-comparisons-have-null-co-kernels}.
\end{proof}

 For any right $R$\+module $N$, the \emph{character module}
$N^+=\Hom_\boZ(N,\boQ/\boZ)$ is a left $R$\+module.
 At the end of this section we collect several simple observations
related to or provable by the passage to the character module.

\begin{lem} \label{character-duality-preserves-reflects-t-c}
 Let $R$ be an arbitrary (nonunital) ring.
 Then a right $R$\+module $N$ is t\+unital if and only if the left
$R$\+module $N^+=\Hom_\boZ(N,\boQ/\boZ)$ is c\+unital.
\end{lem}

\begin{proof}
 One needs to observe that the passage to the character modules
transforms the natural morphism $N\ot_RR\rarrow N$ into the natural
morphism $P\rarrow\Hom_R(R,P)$ for the $R$\+module $P=N^+$.
 Notice the natural isomorphism of abelian groups/left $R$\+modules
$\Hom_\boZ(N\ot_RR,\>\boQ/\boZ)\simeq\Hom_R(R,\Hom_R(N,\boQ/\boZ))$.
\end{proof}

\begin{lem} \label{character-exact-from-t-to-c}
 For any t\+unital ring $R$, the contravariant functor
$N\longmapsto N^+\:\allowbreak(\Modrt R)^\sop\rarrow R\Modlc$ is exact.
\end{lem}

\begin{proof}
 The point is that a complex of $R$\+modules is exact in the quotient
category $R\Modln/R\Modlz$ if and only if its cohomology modules
computed in the abelian category $R\Modln$ are null-modules.
 Consequently (in view of
Theorem~\ref{main-two-category-equivalences-thm}), the same criterion
applies in the abelian categories $\Modrt R$ and $R\Modlc$: a complex
in any one of these categories is exact if and only if, viewed as
a complex, respectively, in $\Modrn R$ or $R\Modln$,
it has null cohomology modules (cf.~\cite[Section~4.6]{Quil}).

 Now it remains to say that the functor
$N\longmapsto N^+\:\allowbreak(\Modrn R)^\sop\rarrow R\Modln$ preserves
the cohomology of complexes, and a right $R$\+module $N$ is null if
and only if the left $R$\+module $N^+$ is null.
\end{proof}

\begin{lem} \label{t-flat-iff-character-dual-t-injective}
 Let $R$ be a t\+unital ring and $F$ be a right $R$\+module.
 Then the left $R$\+module\/ $F^+=\Hom_\boZ(F,\boQ/\boZ)$ is
t\+injective if and only if the right $R$\+module $F$ is t\+flat.
\end{lem}

\begin{proof}
 Let $0\rarrow K\rarrow L\rarrow M$ be a left exact sequence in
$R\Modlt$.
 Then the sequence of abelian groups $\Hom_R(M,F^+)\rarrow
\Hom_R(L,F^+)\rarrow\Hom_R(K,F^+)\rarrow0$ can be obtained by applying
the functor $\Hom_\boZ({-},\boQ/\boZ)$ to the sequence of abelian
groups $0\rarrow F\ot_RK\rarrow F\ot_RL\rarrow F\ot_RM$.
 Hence the former sequence is right exact if and only if the latter
one is left exact.
\end{proof}

\begin{cor} \label{t-flat-cor}
 Let $R$ be a t\+unital ring and $F$ be a t\+unital $R$\+module.
 Then $F$ is t\+flat if and only if it is a flat nonunital $R$\+module,
or in other words, if and only if $F$ is a flat unital
$\widetilde R$\+module.
\end{cor}

\begin{proof}
 Let $F$ be a t\+unital right $R$\+module.
 By Lemma~\ref{character-duality-preserves-reflects-t-c}, \,$F^+$ is
a c\+unital left $R$\+module.
 Lemma~\ref{t-flat-iff-character-dual-t-injective} tells that $F$
is t\+flat if and only if $F^+$ is t\+injective.
 According to Corollary~\ref{t-injective-cor}, \,$F^+$ is t\+injective
if and only if it is an injective object of $R\Modln$, or in other
words, if and only if $F^+$ is an injective left
$\widetilde R$\+module.
 The latter condition is well-known to hold if and only if $F$ is
a flat right $\widetilde R$\+module.
\end{proof}

\begin{prop} \label{c-projective-is-t-flat}
 Let $R$ be a t\+unital ring.
 Then any c\+projective $R$\+module is t\+flat.
\end{prop}

\begin{proof}
 Let $Q$ be a c\+projective left $R$\+module.
 For any left exact sequence $0\rarrow K\rarrow L\rarrow M$ in
$\Modrt R$, the sequence of character modules $M^+\rarrow L^+\rarrow
K^+\rarrow0$ is right exact in $R\Modlc$ by
Lemmas~\ref{character-duality-preserves-reflects-t-c}
and~\ref{character-exact-from-t-to-c}.
 Now the sequence of abelian groups $\Hom_R(Q,M^+)\rarrow
\Hom_R(Q,L^+)\rarrow\Hom_R(Q,K^+)\rarrow0$ can be obtained by applying
the functor $\Hom_\boZ({-},\boQ/\boZ)$ to the sequence of abelian
groups $0\rarrow K\ot_RQ\rarrow L\ot_RQ\rarrow M\ot_RQ$.
 Since the former sequence is right exact by assumption, it follows
that the latter one is left exact.

 Alternatively, one could assume without loss of generality that $Q$
is t\+unital (using Propositions~\ref{c-projective-prop}
and~\ref{t-flat-prop}), and compare Corollary~\ref{c-projective-cor}
with Corollary~\ref{t-flat-cor}.
\end{proof}

 Let $R$ be a t\+unital ring.
 Then the abelian category $R\Modlt\simeq R\Modln/R\Modlz\simeq R\Modlc$
is Grothendieck by Proposition~\ref{grothendieck-with-exact-products}.
 Hence there are enough injective objects in this abelian category
(cf.\ a more general result in~\cite[Lemma~5.4]{Quil}).

 A beautiful construction of Quillen provides (a more general version
of) the following lemma, claiming existence of enough t\+flat
modules.

\begin{lem}[\cite{Quil}]
 Let $R$ be a t\+unital ring and $M$ be a t\+unital $R$\+module.
 Then $M$ is a quotient module of a t\+flat t\+unital $R$\+module.
 In fact, $M$ is the cokernel of a morphism of t\+flat t\+unital
$R$\+modules.
\end{lem}

\begin{proof}
 Notice that the inclusion $R\Modlt\rarrow R\Modln$ preserves
epimorphisms and cokernels; so the assertions are unambiguous.
 The first assertion is~\cite[Lemma~2.4]{Quil}; the second one
is~\cite[Proposition~2.6\,(a)\,$\Rightarrow$\,(b)]{Quil}.
\end{proof}

 On the other hand, the discussion in~\cite[Section~2.5]{Quil},
based on Kaplansky's theorem that all projective modules over
a local ring are free~\cite[Theorem~2]{Kap}, shows that there
\emph{need not} be enough projective objects in $R\Modln/R\Modlz$.
 The following example confirms that this phenomenon occurs for
t\+unital commutative rings~$R$.

\begin{ex}
 Let $R\subset\widetilde R$ be the rings from
Example~\ref{commutative-t-unital-not-s-unital}.
 Choose a prime number~$p$, and put $P=p\boZ\oplus R\subset
\boZ\oplus R=\widetilde R$.
 Then $P$ is a maximal ideal in~$\widetilde R$, since the quotient
ring $\widetilde R/P\simeq\boZ/p\boZ$ is a field.
 The complement $\widetilde R\setminus P$ consists of all the elements
$n+r\in\boZ\oplus R$ such that $n$~is not divisible by~$p$.

 Consider the localizations $\widetilde T'=\widetilde R_{(P)}=
S^{-1}\widetilde R$ and $T=R\widetilde R_{(P)}=S^{-1}R$.
 Then $T$ is an ideal in the commutative local ring~$\widetilde T'$.
 The ring $\widetilde T'$ differs from the unitalization $\widetilde T=
\boZ\oplus T$ of the ring~$T$; in fact, one has $\widetilde T'=
\boZ_{(p)}\oplus T$, where $\boZ_{(p)}$ is the localization of the ring
of integers at the maximal ideal $(p)\subset\boZ$.
 There is a natural injective ring homomorphism $\widetilde T\rarrow
\widetilde T'$ acting by the identity map on~$T$.

 Let us show that any c\+projective $T$\+module is zero.
 By Proposition~\ref{c-projective-prop}, it suffices to check that
any c\+projective t\+unital $T$\+module $F$ is zero.
 It is clear from the definition that any t\+unital $T$\+module is
a $\widetilde T'$\+module (since $\widetilde T'$ is the localization
of $\widetilde T$ at $S'=\boZ\setminus{(p)}$ and $T$ is
a $\widetilde T'$\+module; cf.~\cite[Proposition~9.2]{Quil}).
 By Corollary~\ref{c-projective-cor}, \,$F$ is a projective
$\widetilde T$\+module.
 Hence $F\simeq\widetilde T'\ot_{\widetilde T}F$ is also a projective
$\widetilde T'$\+module.
 Since $\widetilde T'$ is a local ring, Kaplansky's theorem tells
that $F$ is a free $\widetilde T'$\+module.
 But a nonzero free $\widetilde T'$\+module cannot be t\+unital
over~$T$. 
\end{ex}

\Section{s-Unital and t-Unital Homomorphisms of Rings}
\label{homomorphisms-secn}

 We start with a brief discussion of s\+unital ring homomorphisms
before passing to the somewhat more complicated t\+unital case.

 Let $f\:K\rarrow R$ be a homomorphism of nonunital rings.
 We will say that the homomorphism~$f$ is \emph{left s\+unital}
if $R$ an s\+unital left $K$\+module in the module structure
induced by~$f$.
 Similarly, we say that $f$~is \emph{right s\+unital} if
$f$~makes $R$ an s\+unital right $K$\+module.
 Finally, we say that $f$~is \emph{s\+unital} if it is both
left and right s\+unital.

\begin{lem} \label{s-unital-homomorphisms}
 Let $f\:K\rarrow R$ be a homomorphism of nonunital rings.
 Assume that $f$~is s\+unital (respectively, left s\+unital,
right s\+unital).
 Then the ring $R$ is s\+unital (resp., left s\+unital,
right s\+unital).
\end{lem}

\begin{proof}
 Let $r\in R$ be an element.
 The homomorphism~$f$ being left s\+unital means existence of
an element $e\in K$ such that $f(e)r=r$ in~$R$.
 The ring $R$ being left s\+unital means existence of
an element $g\in R$ such that $gr=r$.
 In order to show that the former implies the latter, it remains
to take $g=f(e)$.
\end{proof}

 The next lemma is even more obvious.

\begin{lem} \label{restriction-of-scalars-reflects-s-unitality-lemma}
 Let $K\rarrow R$ be a homomorphism of nonunital rings.
 Then any $R$\+module that is s\+unital as a $K$\+module is
also s\+unital as an $R$\+module. \qed
\end{lem}

\begin{cor} \label{restriction-scalars-preserves-reflects-s-unitality}
 Let $f\:K\rarrow R$ be a left s\+unital homomorphism of nonunital
rings.
 Then a left $R$\+module is s\+unital as an $R$\+module if and only if
it is s\+unital as a $K$\+module.
\end{cor}

\begin{proof}
 The ```if'' implication is provided by
Lemma~\ref{restriction-of-scalars-reflects-s-unitality-lemma}.
 To prove the ``only if'', let $M$ be an s\+unital left $R$\+module
and $m\in M$ be an element.
 By assumption, there exist an element $g\in R$ such that $gm=m$ in $M$
and an element $e\in K$ such that $f(e)g=g$ in~$R$.
 Now it is clear that $f(e)m=f(e)gm=gm=m$ in~$M$.
\end{proof}

 We will say that a ring homomorphism $f\:K\rarrow R$ is \emph{left
t\+unital} if $R$ a t\+unital left $K$\+module in the module structure
induced by~$f$.
 Similarly, we say that $f$~is \emph{right t\+unital} if
$f$~makes $R$ a t\+unital right $K$\+module.
 At last, we say that $f$~is \emph{t\+unital} if it is both
left and right t\+unital.

\begin{lem}
\textup{(a)} If a ring $K$ is left s\+unital, then a ring homomorphism
$f\:K\rarrow R$ is left t\+unital if and only if $f$~is left s\+unital.
\par
\textup{(b)} If a ring $K$ is right s\+unital, then a ring homomorphism
$f\:K\rarrow R$ is right t\+unital if and only if $f$~is
right s\+unital.
\end{lem}

\begin{proof}
 Follows from Corollary~\ref{s-unital-iff-t-unital-for-modules}.
\end{proof}

 Clearly, if $KR=R$ or $RK=R$ (which means $f(K)R=R$ or $Rf(K)=R$,
respectively), then $R^2=R$.
 The following proposition, providing a direct analogue of
Lemma~\ref{s-unital-homomorphisms} for t\+unitality, is a bit
more involved.

\begin{prop} \label{t-unital-homomorphisms}
 Let $f\:K\rarrow R$ be a homomorphism of nonunital rings.
 Assume that $f$~is \emph{either} left t\+unital \emph{or} right
t\+unital.
 Then the ring $R$ is t\+unital.
\end{prop}

\begin{proof}
 Assuming that the multiplication map $K\ot_KR\rarrow R$ is
an isomorphism, we need to prove that so is the multiplication map
$R\ot_RR\rarrow R$.

 Indeed, consider the initial fragment of a relative bar-complex
\begin{equation} \label{bar-complex-fragment}
 R\ot_KR\ot_KR\overset\partial\lrarrow R\ot_KR
 \overset\partial\lrarrow R\lrarrow0
\end{equation}
with the differentials $\partial(r'\ot r''\ot r''') = r'r''\ot r'''
- r'\ot r''r'''$ and $\partial(r'\ot r'') = r'r''$ for all
$r'$, $r''$, $r'''\in R$.
 We have to show that the short sequence~\eqref{bar-complex-fragment}
is right exact.
 This is obviously equivalent to the desired isomorphism
$R\ot_RR\simeq R$.

 Actually, the short sequence of left
$K$\+modules~\eqref{bar-complex-fragment} is split right exact
(and in fact, the whole bar-complex is split exact as a complex
of left $K$\+modules).
 Let us construct a contracting homotopy~$h$.
 Its component $h_0\:R\rarrow R\ot_KR$ is defined as the composition
$R\simeq K\ot_KR\rarrow R\ot_KR$ of the inverse map to the natural
isomorphism $K\ot_KR\rarrow R$ and the map $f\ot_KR\:K\ot_KR\rarrow
R\ot_KR$ induced by the $K$\+$K$\+bimodule morphism $f\:K\rarrow R$.
 Similarly, the component $h_1\:R\ot_KR\rarrow R\ot_KR\ot_KR$ is
defined as the composition $R\ot_KR\rarrow K\ot_KR\ot_KR\rarrow
R\ot_KR\ot_KR$ of the inverse map to the natural isomorphism
$K\ot_KR\ot_KR\rarrow R\ot_KR$ and the map $f\ot_KR\ot_KR\:
K\ot_KR\ot_KR\rarrow R\ot_KR\ot_KR$ induced by the morphism~$f$.
 It is straightforward to check that $\partial h+h\partial=\id\:
R\ot_KR\rarrow R\ot_KR$, proving the desired exactness at
the term $R\ot_KR$.

 Alternatively, one could obtain the desired result (for right
t\+unital ring homomorphisms~$f$) as a particular case of the next
Proposition~\ref{restriction-of-scalars-t-c-unitality}(a).
\end{proof}

 Arguing similarly to the proof of
Proposition~\ref{t-unital-homomorphisms} and using initial fragments
of relative bar/cobar-complexes for $R$\+modules, one can show that,
if a ring homomorphism $f\:K\rarrow R$ is right t\+unital, then
any right $R$\+module that is t\+unital as a $K$\+module is also
t\+unital as an $R$\+module, and any left $R$\+module that is
c\+unital as a $K$\+module is also c\+unital as an $R$\+module.
 A more direct approach allows to obtain the same conclusions under
more relaxed assumptions.

\begin{prop} \label{restriction-of-scalars-t-c-unitality}
 Let $f\:K\rarrow R$ be a homomorphism of nonunital rings.
 Assume that $RK=R$, or more generally, $KR\subset RK$ in~$R$.
 Then \par
\textup{(a)} any right $R$\+module that is t\+unital as a $K$\+module
is also t\+unital as an $R$\+module; \par
\textup{(b)} any left $R$\+module that is c\+unital as a $K$\+module
is also c\+unital as an $R$\+module. \par
 In particular, if the ring $R$ is commutative, then
the assertions~\textup{(a\+-b)} hold for any homomorphism~$f$.
\end{prop}

\begin{proof}
 Part~(b): Let $P$ be a left $R$\+module that is c\+unital as
a left $K$\+module.
 Then, first of all, the map $P\rarrow\Hom_K(K,P)$ is injective;
this means that no nonzero element of $P$ is annihilated by~$K$.
 It follows that no nonzero element of $P$ is annihilated by~$R$,
either; so the map $P\rarrow\Hom_R(R,P)$ is injective.

 To prove surjectivity of the latter map, consider a left
$R$\+module morphism $g\:R\rarrow P$.
 The composition $K\rarrow R\rarrow P$ is a left $K$\+module
morphism $gf\:K\rarrow P$.
 Since $P$ is c\+unital over~$K$, there exists an element $p\in P$
such that $gf(k)=kp$ for all $k\in K$.
 Consider the $R$\+module morphism $h\:R\rarrow P$ given by
the formula $h(r)=g(r)-rp$ for all $r\in R$.
 Then the composition $hf\:K\rarrow P$ vanishes; and we have to
prove that the whole map~$h$ vanishes.

 Indeed, for all $r\in R$ and $k\in K$ we have $h(rf(k))=rh(f(k))=0$
in~$P$; so $h|_{RK}=0$.
 By assumption, we have $KR\subset RK\subset R$; hence $h|_{KR}=0$.
 This means that for any $r\in R$ and $k\in K$ we have $kh(r)=
h(f(k)r)=0$ in~$P$.
 Thus the element $h(r)\in P$ is annihilated by $K$, and we can
conclude that $h(r)=0$.

 Part~(a): let $N$ be a right $R$\+module that is t\+unital as
a right $K$\+module.
 Then, first of all, the map $N\ot_KK\rarrow N$ is surjective;
so $NK=N$, and it follows that $NR=N$, or in other words, the map
$N\ot_RR\rarrow N$ is surjective.

 To prove injectivity of the latter map, consider the composition
$N\ot_KK\rarrow N\ot_RR\rarrow N$.
 By assumption, the map $N\ot_KK\rarrow N$ is an isomorphism.
 So it suffices to show that the map $N\ot_KK\rarrow N\ot_RR$ is
surjective.
 We have $NK=N$; so any element $n\in N$ has the form $n=\sum_{i=1}^j
n_ik_i$ for some $n_i\in N$ and $k_i\in K$.
 Hence, for any $r\in R$, we have $n\ot_Rr=\sum_{i=1}^jn_ik_i\ot_Rr
=\sum_{i=1}^jn_i\ot_Rk_ir$.

 By assumption, $KR\subset RK$; so, for every $1\le i\le j$, we have
$k_ir=\sum_{u=1}^{v_i}r_{i,u}k'_{i,u}$ for some $r_{i,u}\in R$ and
$k'_{i,u}\in K$.
 Thus $n\ot_Rr=\sum_{i=1}^jn_i\ot_Rk_ir=\sum_{i=1}^j\sum_{u=1}^{v_i}
n_i\ot_Rr_{i,u}k'_{i,u}=\sum_{i=1}^j\sum_{u=1}^{v_i}n_ir_{i,u}
\ot_Rk'_{i,u}$, and the latter element belongs to the image of
the map $N\ot_KK\rarrow N\ot_RR$, as desired.

 Alternatively, one can deduce part~(a) from part~(b) using
Lemma~\ref{character-duality-preserves-reflects-t-c}.
\end{proof}

\begin{cor}
 Let $f\:K\rarrow R$ be a right t\+unital homomorphism of nonunital
rings.  Then \par
\textup{(a)} a right $R$\+module is t\+unital as an $R$\+module if and
only if it is t\+unital as a $K$\+module; \par
\textup{(b)} a left $R$\+module is c\+unital as an $R$\+module if and
only if it is c\+unital as a $K$\+module.
\end{cor}

\begin{proof}
 If $f$~is right t\+unital, then $RK=R$.
 So Proposition~\ref{restriction-of-scalars-t-c-unitality} is
applicable, providing the ``if'' implications in (a) and~(b).

 ``Only if'' in~(a): for any right $R$\+module $N$,
the right $R$\+module $N\ot_RR$ is t\+unital as a right $K$\+module
by Lemma~\ref{t-unital-tensor-product-lemma}(b) applied to $B=R'=R$
and $R''=K$ (since $R$ is a t\+unital right $K$\+module).
 If $N$ is a t\+unital right $R$\+module, then $N\simeq N\ot_RR$.

 ``Only if'' in~(b): for any left $R$\+module $P$,
the left $R$\+module $\Hom_R(R,P)$ is c\+unital as a left $K$\+module
by Lemma~\ref{t-unital-c-unital-Hom-lemma} applied to $B=R'=R$
and $R''=K$ (since $R$ is a t\+unital right $K$\+module).
 If $P$ is a c\+unital left $R$\+module, then $P\simeq\Hom_R(R,P)$.
\end{proof}

\begin{ex} \label{t-c-unitality-restriction-of-scalars-counterex}
 (1)~The following counterexample shows that the assertion of
Proposition~\ref{restriction-of-scalars-t-c-unitality}(a) is
\emph{not} true for an arbitrary ring homomorphism~$f$ in general.

 Let $\widetilde R=\boZ[x]$ and $\widetilde S=\boZ[y]$ be two
rings of polynomials in one variable with integer coefficients,
and let $\widetilde T=\boZ\{x,y\}$ be the ring of noncommutative
polynomials in two variables $x$, $y$, i.~e., the associative
unital ring freely generated by its unital subrings
$\widetilde R$ and $\widetilde S$.
 Let $R=(x)=xR\subset\widetilde R$ be the ideal spanned by~$x$
in $\widetilde R$ and $T=(x,y)\subset\widetilde T$ be the ideal
spanned by $x$ and~$y$ in~$\widetilde T$.
 So $R$ is a subring in~$T$.

 The left $T$\+module $T$ decomposes as a direct sum of two free unital
$\widetilde T$\+modules with one generator
$T=\widetilde Tx\oplus\widetilde Ty$.
 Consequently, one has $N\ot_TT\simeq N\oplus N$ for any right
$R$\+module $N$, and the map $N\ot_TT\rarrow N$ can be identified
with the map $N\oplus N\xrightarrow{(x,y)}N$.
 On the other hand, $R=\widetilde Rx$ is a free unital
$\widetilde R$\+module with one generator.
 So for any $R$\+module $N$ one has $N\ot_RR\simeq N$, and the map
$N\ot_RR\rarrow N$ can be identified with the map
$N\overset x\rarrow N$.

 Now let $N$ be any nonzero unital right $\widetilde T$\+module where
$x$~acts by an invertible map $x\:N\simeq N$.
 Then it is clear that $N\oplus N\xrightarrow{(x,y)}N$ is a surjective,
noninjective map.
 So $N$ is \emph{not} a t\+unital $T$\+module, even though it is
a t\+unital module over the subring $R\subset T$.

 (2)~Similarly one constructs a counterexample showing that
the assertion of
Proposition~\ref{restriction-of-scalars-t-c-unitality}(b) does not
generalize to arbitrary ring homomorphisms~$f$.

 Let $R\subset T$ be the same pair of rings as in~(1).
 Then, for any left $T$\+module $P$, the direct sum decomposition
$T=\widetilde Tx\oplus\widetilde Ty$ of the left $T$\+module $T$
implies an isomorphism $\Hom_T(T,P)\simeq P\oplus P$.
 The map $P\rarrow\Hom_T(T,P)$ can be identified with
the map $P\xrightarrow{(x,y)}P\oplus P$.
 On the other hand, since $R=\widetilde Rx$ is a free unital
$\widetilde R$\+module, for any $R$\+module $P$ one has
$\Hom_R(R,P)\simeq P$, and the map $P\rarrow\Hom_R(R,P)$
can be identified with the map $P\overset x\rarrow P$.

 Now let $P$ be any nonzero unital left $\widetilde T$\+module where
$x$~acts by an invertible map $x\:P\simeq P$.
 Then $P\xrightarrow{(x,y)}P\oplus P$ is an injective,
nonsurjective map.
 So $P$ is \emph{not} a c\+unital $T$\+module, even though it is
a c\+unital module over the subring $R\subset T$.
\end{ex}

\begin{ex}
 This is an improved version of
Example~\ref{t-c-unitality-restriction-of-scalars-counterex} providing
a (ring, subring) pair $R\subset T$ such that \emph{both the rings $R$
and $T$ are t\+unital}, but the assertions of
Proposition~\ref{restriction-of-scalars-t-c-unitality} still do not
hold for the inclusion map $f\:R\rarrow T$.

 Let $\widetilde R=\boZ[x^p\mid p\in\boQ_{\ge0}]$ and
$\widetilde S=\boZ[y^q\mid q\in\boZ_{\ge0}]$ be the two rings from
Example~\ref{rational-power-polynomials-two-variables-t-unital}.
 Let $\widetilde T$ be the associative unital ring freely generated
by its unital subrings $\widetilde R$ and~$\widetilde S$.
 So, as an abelian group, $\widetilde T$ has a basis consisting of
words in the alphabet~$x^p$ and~$y^q$, with $p$, $q\in\boQ_{>0}$
and the powers of~$x$ alternating with the powers of~$y$ in
the sequences forming the words.
 Let $T\subset\widetilde T$ be the subgroup spanned by nonempty words,
or in other words, the ideal in $\widetilde T$ spanned by $x^p$
and~$y^q$ with $p$, $q>0$.

 Example~\ref{commutative-t-unital-not-s-unital} tells that
the ring $R=\boZ[x^p\mid p\in\boQ_{>0}]\subset\widetilde R$ is
t\+unital.
 A similar explicit argument proves that the ring $T$ is t\+unital.
 Recall also the notation
$S=\boZ[y^q\mid q\in\boQ_{>0}]\subset\widetilde S$.
 The left $T$\+module $T$ decomposes as a direct sum of two modules
$T=\bigcup_{n\ge1}\widetilde Tx^{1/n}\oplus\bigcup_{n\ge1}
\widetilde Ty^{1/n}\simeq\widetilde T\ot_{\widetilde R}R\oplus
\widetilde T\ot_{\widetilde S}S$.

 (1)~Let $N$ be any nonzero unital right $\widetilde T$\+module
where all the elements $x^p$ and~$y^q$ act by invertible operators.
 Then one has $N\ot_RR\simeq N$, and a natural choice of the latter
isomorphism identifies the map $N\ot_RR\rarrow N$ with the identity
map~$\id_N$.
 One also has $N\ot_TT\simeq N\oplus N$, and the map $N\ot_TT\rarrow N$
can be identified with the diagonal projection $N\oplus N
\xrightarrow{(1,1)}N$.
 So $N$ is \emph{not} a t\+unital $T$\+module, but $N$ is
a t\+unital module over the subring $R\subset T$.

 (2)~Let $P$ be any nonzero unital left $\widetilde T$\+module
where all the elements $x^p$ and~$y^q$ act by invertible operators.
 Then one has $\Hom_R(R,P)\simeq P$, and a natural choice of
the latter isomorphism identifies the map $P\rarrow\Hom_R(R,P)$
with the identity map~$\id_P$.
 One also has $\Hom_T(T,P)\simeq P\oplus P$, and the map
$P\rarrow\Hom_T(T,P)$ can be identified with the diagonal embedding
$P\xrightarrow{(1,1)}P\oplus P$.
 So $P$ is \emph{not} a c\+unital $T$\+module, but $P$ is
a c\+unital module over the subring $R\subset T$.
\end{ex}

\end{document}